\documentclass[oneside,a4paper,reqno]{amsart}
\usepackage{pdfsync, enumitem}
\usepackage{amsmath,bm}
\usepackage{mathtools}
\usepackage{graphicx}
\usepackage[all,cmtip]{xy}
\usepackage{tikz-cd}
\usepackage{mathrsfs}
\usepackage{stackrel}
\usepackage{adjustbox}
\usepackage{multicol}
\usepackage{amsmath, amsthm, amscd, amssymb, latexsym, eucal}

\def\serieslogo@{} \def\@setcopyright{} \makeatother
\usepackage{fullpage}
\usepackage[margin=2.5cm, top=2.5cm, bottom=2.5cm, footskip=1cm]{geometry}

\usepackage[colorinlistoftodos]{todonotes}

\usepackage{cite}
\usepackage{hyperref}
\usepackage{color}
\usepackage{quiver}

\makeatletter
\renewcommand*\env@matrix[1][c]{\hskip -\arraycolsep
	\let\@ifnextchar\new@ifnextchar
	\array{*\c@MaxMatrixCols #1}}
\makeatother

\numberwithin{equation}{section}
\newtheorem{thm}{Theorem}[section]
\newtheorem*{main-thm}{Theorem}
\newtheorem*{Auslander-thm}{Auslander's Theorem}

\newtheorem{cor}[thm]{Corollary}
\newtheorem{lem}[thm]{Lemma}
\newtheorem{prop}[thm]{Proposition}

\newtheorem*{mainthm}{Main~Theorem}

\theoremstyle{definition}
\newtheorem{defn}[thm]{Definition}
\newtheorem{rem}[thm]{Remark}
\newtheorem{exmp}[thm]{Example}

\newtheorem*{connot}{Notation and Conventions}

\newcommand\padova{%
\mathrel{\ooalign{\hss{\scalebox{0.5}{$\longleftrightarrow$}}\hss\cr%
\kern0.9ex\raise0.55ex\hbox{\scalebox{0.5}{$\boldsymbol{\bm{\vert}}$}}}}}
\newcommand\hfpadova{%
\mathrel{\ooalign{\hss{\scalebox{0.5}{$\longleftrightarrow$}}\hss\cr%
\kern0.9ex\raise0.55ex\hbox{\scalebox{0.5}{$\boldsymbol{\mathsf{h}\mathsf{f}\ \ }$}}}}}
\newcommand\rpadova{%
\mathrel{\ooalign{\hss{\scalebox{0.5}{$\longrightarrow$}}\hss\cr%
\kern0.9ex\raise0.55ex\hbox{\scalebox{0.5}{$\boldsymbol{\bm{\vert}}$}}}}}
\newcommand\lpadova{%
\mathrel{\ooalign{\hss{\scalebox{0.5}{$\longleftarrow$}}\hss\cr%
\kern0.9ex\raise0.55ex\hbox{\scalebox{0.5}{$\boldsymbol{\bm{\vert}}$}}}}}

\def\g{\gamma}
\def\d{\delta}
\def\e{\epsilon}

\def\i{\iota}

\DeclareMathOperator*{\Ker}{\mathsf{Ker}}

\DeclareMathOperator{\pd}{\mathsf{pdim}}
\DeclareMathOperator{\fd}{\mathsf{fdim}}
\DeclareMathOperator*{\id}{\mathsf{idim}}

\DeclareMathOperator*{\lsmod}{\!-\mathsf{mod}}
\DeclareMathOperator*{\lMod}{\!-\mathsf{Mod}}

\DeclareMathOperator*{\lproj}{\!-\mathsf{proj}}

\DeclareMathOperator*{\Inj}{\mathsf{Inj}-\!}
\DeclareMathOperator*{\lInj}{\!-\mathsf{Inj}}

\DeclareMathOperator*{\Proj}{\mathsf{Proj}-\!}
\DeclareMathOperator*{\lProj}{\!-\mathsf{Proj}}

\DeclareMathOperator{\Hom}{\mathsf{Hom}}

\DeclareMathOperator*{\Tor}{\mathsf{Tor}}

\usepackage{stackengine}

\newsavebox{\proofbox}
\savebox{\proofbox}{\begin{picture}(7,7)%
	\put(0,0){\framebox(7,7){}}\end{picture}}

\newcommand{\comp}{\mathop{\raisebox{+.3ex}{\hbox{$\scriptstyle\circ$}}}}

\usepackage{latexsym}
\usepackage{pstricks}
\usepackage{comment}

\usepackage[greek,english]{babel}

\begin{document}

\title{Equivariant Cleft Extensions and Singular Equivalences}

\author[karakikes]{Miltiadis Karakikes}
\address{Department of Mathematics, National and Kapodistrian University of Athens, Panepistimioupolis, 15784 Athens, Greece}
\email{miltoskar@math.uoa.gr}

\subjclass[2020]{16E30, 16E35, 18E10, 16G10, 18G20, 18G80}


\keywords{cleft extensions, trivial extensions, extensions of rings, equivariant category, singularity category}
\thanks{\textbf{Acknowledgements:} The author is grateful to P.~Kostas for valuable discussions on the subject of cleft extensions and for feedback on the manuscript. Thanks are also due to C.~Psaroudakis for his helpful comments on earlier drafts. The author is also deeply grateful to the anonymous referee for a careful reading of the manuscript and for numerous insightful comments that improved both the accuracy and the presentation of this work.
}

\begin{abstract}
We study the equivariant lifting of cleft extensions of abelian categories and its impact on singularity categories. Specifically, we establish the framework for lifting a cleft extension to a $G$-equivariant cleft extension. Furthermore, we prove that a restriction functor associated to a cleft extension induces a singular equivalence if and only if its equivariant counterpart does. As a concrete application, we demonstrate that the skew group ring of a $G$-equivariant $\theta$-extension is isomorphic to a $\widehat{\theta}$-extension of the base skew group ring, allowing us to lift singular equivalences for these structures.
\end{abstract}

\date{\today}
\maketitle
\setcounter{tocdepth}{1}

\section{Introduction}

A \emph{cleft extension of a ring} $\Gamma$ consists of a triple $(\Lambda, \pi, \iota)$, where $\Lambda$ is a ring and $\pi \colon \Lambda \to \Gamma$ and $\iota \colon \Gamma \to \Lambda$ are ring homomorphisms such that $\pi \circ \iota = \mathsf{Id}_\Gamma$. Any cleft extension of rings is equivalent to a $\theta$-extension of rings (Definition~\ref{def: theta}), introduced by Marmaridis \cite{marmaridis}. This class of extensions is broad and includes  trivial extension rings (in particular, triangular matrix rings) and tensor algebras, which are of interest to us in this article. For a complete list of examples that can be viewed as cleft extensions of rings we refer the reader to \cite[Introduction]{beligiannis}.

Cleft extensions of abelian categories, formalized and studied extensively by Beligiannis \cite{beligiannis, beligiannis2}, serve as categorical counterparts to cleft extensions of rings. In particular, any cleft extension of module categories is Morita equivalent to a cleft extension induced by a $\theta$-extension \cite[Proposition~3.3]{beligiannis}. In this paper, we utilize $\theta$-extensions because the definition is more concrete and allows us to handle everything expressed intrinsically in terms of the ring $\Gamma$.

In recent decades, singularity categories have become objects of extensive study because they encode information regarding the singularities of a variety (in geometric settings) or the failure of a ring to have finite global dimension (in algebraic settings). A fundamental problem in this area is identifying whether two abelian categories with enough projectives have equivalent singularity categories—a phenomenon known as a singular equivalence. Recently, singular equivalences of categories arising from cleft extensions were studied by Kostas \cite{kostas}.

The starting point of this work is the following set of questions:
Assume a finite group $G$ acts by ring automorphisms on a finite dimensional $k$-algebra $\Gamma$, where $k$ is a field, and suppose $\Gamma$ is part of a $\theta$-extension $ \Gamma \! \ltimes_\theta \! M$. Does the group action extend to the $\theta$-extension? How is the skew group ring $(\Gamma \! \ltimes_\theta \! M)\#G$ related to $\Gamma \# G$? Under what conditions does a singular equivalence of a $\theta$-extension of a finite dimensional algebra over a field lift to a singular equivalence of the corresponding skew group algebra induced by a finite group?
The main result of this work provides definitive answers to these structural and homological questions for $G$-equivariant $\theta$-extensions $\Lambda = \Gamma \! \ltimes_\theta \! M$ (Definition~\ref{def: G-equivariant theta extensions}).
Note that such group actions are also an object of interest in several recent works \cite{RR, BeckOber, ChenWang, Demonet, kkp}.

\begin{mainthm} \label{main theorem}
    Let $\Gamma$ be a finite dimensional algebra over a field $k$ and $M$ a $\Gamma$-bimodule that is finitely generated on both sides. Let also $G$ be a finite group acting on $\Gamma$ by ring automorphisms. Consider a $G$-equivariant $\theta$-extension $\Lambda = \Gamma \! \ltimes_\theta \! M$. Then the $G$-action extends to $\Lambda$ by automorphisms. Moreover, the skew group algebra $\Lambda \# G $ is isomorphic to the $\widehat{\theta}$-extension $\Gamma \#G \ltimes_{\widehat{\theta}} M \# G$ of the skew group algebra $\Gamma \# G$ by the induced bimodule $M \# G$, equipped with an induced homomorphism $\widehat{\theta} \colon M\# G \otimes_{\Gamma\#G} M\#G \to M\#G$.

    Now, assume that $|G|$ is invertible in $\Gamma$, $M$ is tensor nilpotent, $\pd {_\Gamma M} < \infty$, $\pd {M _\Gamma} < \infty$, and $\Tor^\Gamma_i (M^{\otimes j} , M)=0$ for all $i,j \geq 1$. Then the following are equivalent:
    \begin{itemize}
        \item[\textnormal{(i)}] The restriction functor $\mathsf{e}\colon (\Gamma \! \ltimes_\theta \! M)\lMod  \to \Gamma\lMod $ induces a singular equivalence.
        \item[\textnormal{(ii)}] The restriction functor $\mathsf{e}\#G \colon (\Gamma \#G \ltimes_{\widehat{\theta}}  M \# G)\lMod  \to \Gamma \# G\lMod $ induces a singular equivalence.
    \end{itemize}
\end{mainthm}

In particular, if $\pd {_\Gamma M _\Gamma} < \infty$, then condition (i) holds by \cite[Lemma~3.5]{Dalezios} coupled with \cite[Corollary~6.13]{kostas}. We obtain the following corollary, where $\mathsf{D_{sg}}(\Gamma)$ denotes the singularity category of $\Gamma$.

\begin{cor} \label{main: cor}
Let $\Gamma$ be a finite dimensional algebra over a field $k$ and $M$ a $\Gamma$-bimodule that is finitely generated on both sides, and let $G$ be a finite group acting on $\Gamma$ by ring automorphisms. Consider a $G$-equivariant $\theta$-extension $\Lambda = \Gamma \ltimes_\theta M$. Assume that $|G|$ is invertible in $\Gamma$, $M$ is tensor nilpotent, $\pd {_\Gamma M _\Gamma} < \infty$, and $\Tor^\Gamma_i (M^{\otimes j} , M)=0$ for all $i,j \geq 1$. Then there exist singular equivalences:
\[
\mathsf{D_{sg}}(\Gamma \! \ltimes_\theta \! M)  \to \mathsf{D_{sg}}(\Gamma) \quad \textnormal{and} \quad \mathsf{D_{sg}}(\Gamma\# G \ltimes_{\widehat{\theta}} M\#G)  \to \mathsf{D_{sg}}(\Gamma\#G).
\]
\end{cor}

To prove the Main Theorem, we establish several general results regarding equivariant cleft extensions of abelian categories and their singularity categories. Equivariant categories arise as a type of categorical quotient by a finite group \cite{Elagin, BeckOber, git, ChaoSun, XiaoWuChen2, Mumford}; such group actions on module categories are induced by actions by automorphisms on rings \cite{Demonet,RR}. We establish sufficient conditions for a cleft extension to lift into the equivariant setting (Proposition~\ref{prop:equivariant_cleft_properties}). We generalize known results (e.g.\ \cite[Proposition~3.10]{ChenWang}) to the setting of $\theta$-extensions by demonstrating that a $G$-equivariant $\theta$-extension $\Gamma \ltimes_\theta M$ (Definition~\ref{def: G-equivariant theta extensions}) induces a canonical $\widehat{\theta}$-extension $\Gamma \# G \ltimes_{\widehat{\theta}} M\#G$, which is isomorphic to the skew group ring $(\Gamma \ltimes_\theta M ) \# G$ (Theorem~\ref{thm: main A}).

The interplay between singular equivalences and equivariant categories remains a topic of active interest \cite{kkp, kk, ChaoSun}. Assuming the order $|G|$ is invertible in the relevant categories, it is known that if an exact $G$-functor $\mathsf{e} \colon \mathcal{A} \to \mathcal{B}$ induces a singular equivalence, then $\mathsf{e}^G \colon \mathcal{A}^G \to \mathcal{B}^G$ induces a singular equivalence up to retracts; see \cite[Theorem~6.11]{kkp}.
We prove that if $\mathsf{e} \colon \mathcal{A} \to \mathcal{B}$ is part of the data of a $G$-equivariant cleft extension (Definition~\ref{def: G-lift of cleft}), then $\mathsf{e}^G$ upgrades to an honest singular equivalence, and moreover, the converse also holds (Theorem~\ref{thm: main B}). We remark that in the presence of a cleft extension, the converse holds without assuming the invertibility of $|G|$ (Remark~\ref{rem: converse without |G|}). Finally, as a necessary step toward our results on singularity categories, we generalize several properties of derived functors of exact functors from \cite[Section~5.1]{kkp} to the left derived functors of right exact functors, and dually, to the right derived functors of left exact functors.

\subsection*{\textbf{Structure of the paper.}}
In Section \ref{preliminaries}, we establish the necessary background on cleft extensions, $\theta$-extensions, and equivariant categories. Section~\ref{sec: equivariant lifts} presents our first main results regarding the lift of cleft extensions to the equivariant setting, alongside an analysis of skew group $\theta$-extensions. In Section~\ref{sec: equivariant derived cats}, we establish key facts concerning equivariant derived categories and the derived functors of $G$-functors. Building on this machinery, Section~\ref{sec: equivariant singularity categories} lifts singular equivalences to the equivariant setting. Finally, in Section~\ref{sec: application}, we apply this accumulated theory to the skew group rings of $\theta$-extensions and demonstrate our results through concrete examples.

\begin{connot}
For an abelian category $\mathcal{A}$, we denote by $\Proj\mathcal{A}$ and $\Inj\mathcal{A}$ its full subcategories of projective and injective objects, respectively, and by $\mathsf{K}^?(\mathcal{A})$ (resp.\ $\mathsf{D}^?(\mathcal{A})$) the homotopy (resp.\ derived) category of cochain complexes in $\mathcal{A}$, where $? \in \{ \varnothing, -, +, \mathsf{b} \}$ denotes unbounded, bounded above, bounded below, and bounded complexes, respectively.
For a ring $R$, we denote the category of all left $R$-modules by ${R}\lMod$, and we consider its full subcategories: ${R}\lProj$ (projective modules), ${R}\lproj$ (finitely generated projective modules), $R\lInj $ (injective modules), and ${R} \lsmod$ (finitely presented modules). We say that a positive integer $n$ is \emph{invertible} in an additive category $\mathcal{A}$ if multiplication by $n$ is an automorphism of every object $X$ of $\mathcal{A}$; as a consequence multiplication by $n$ acts as an automorphism on every abelian group $\Hom_\mathcal{A}(X,Y)$. All rings are associative with identity, and $G$ denotes a finite group, written multiplicatively.
\end{connot}

\section{Preliminaries}\label{preliminaries}

\subsection{Cleft extensions of abelian categories}\label{subsec: cleft of abelian}
We begin by recalling the structure of cleft extensions of abelian categories, as developed by Beligiannis \cite{beligiannis, beligiannis2}.

\begin{defn} \label{cleft of abelian}
    A \emph{cleft extension} of an abelian category $\mathcal{B}$ is an abelian category $\mathcal{A}$ together with additive functors:
    \begin{center}
        \begin{tikzcd}
\mathcal{B} \arrow[rr, "\mathsf{i}"] &  & \mathcal{A} \arrow[rr, "\mathsf{e}"] &  & \mathcal{B}, \arrow[ll, "\mathsf{l}"', bend right]
\end{tikzcd}
    \end{center}
    henceforth denoted by $(\mathcal{B},\mathcal{A},\mathsf{i},\mathsf{e},\mathsf{l})$, such that the following conditions are satisfied: 
    \begin{itemize}
        \item[(i)] the functor $\mathsf{e}$ is faithful exact,
        \item[(ii)] the pair $(\mathsf{l},\mathsf{e})$ is an adjoint pair,
        \item[(iii)] there is a natural isomorphism $\mathsf{ei}\simeq \mathsf{Id}_{\mathcal{B}}$ of functors. 
    \end{itemize}
\end{defn}

From the definition, a cleft extension naturally admits important extra structure. There exists a functor $\mathsf{q} \colon \mathcal{A} \to \mathcal{B}$ which is a left adjoint to $\mathsf{i}$, see \cite[Proposition~2.3]{beligiannis}. 
Let $\d \colon \mathsf{Id}_\mathcal{B} \to \mathsf{e}\mathsf{l}$ and $\e \colon \mathsf{l} \mathsf{e} \to \mathsf{Id}_\mathcal{A}$ denote the unit and the counit of the adjunction $(\mathsf{l},\mathsf{e})$, respectively. 
We consider the functor $\mathsf{G} \colon \mathcal{A} \to \mathcal{A}$ defined as $\Ker (\e)$ and the functor $\mathsf{F}\colon \mathcal{B} \to \mathcal{B}$ defined as the composition $\mathsf{e}\mathsf{G}\mathsf{i}$. Note that there exist natural isomorphisms $\mathsf{el}\simeq \mathsf{Id}_{\mathcal{B}}\oplus\mathsf{F}$ and $\mathsf{eG} \simeq \mathsf{Fe}$.

Note that any cleft extension of rings (as defined in the introduction) gives rise to a cleft extension of module categories induced by the restriction and induction functors along the ring homomorphisms. 

\begin{rem} \label{rem: eta extensions}
    We remark that cleft extensions of abelian categories are equivalent to $\eta$-extensions of abelian categories (see \cite[Theorem~2.6]{beligiannis}), which were introduced in \cite{marmaridis} as a natural generalization of trivial extensions studied extensively in \cite{FGR, PR} and serve as a categorical counterpart to $\theta$-extensions; see also \cite[Proposition~5.2]{marmaridis}. Given an additive right exact endofunctor $\mathsf{F}\colon \mathcal{B} \to \mathcal{B}$ on an abelian category $\mathcal{B}$ and a natural transformation $\eta\colon \mathsf{F}^2 \to \mathsf{F}$ that is associative (i.e.\ $\eta \circ \mathsf{F}\eta = \eta \circ \eta\mathsf{F}$), we define the \emph{$\eta$-extension} of $\mathcal{B}$ by $\mathsf{F}$ to be the additive category $\mathcal{B} \ltimes_\eta \mathsf{F}$ whose objects are pairs $(X,u)$ where $X$ is an object of $\mathcal{B}$ and $u\colon \mathsf{F}(X)\rightarrow X$ is a morphism in $\mathcal{B}$ such that $u\circ \mathsf{F}u=u\circ \eta_X$, and whose morphisms $f\colon (X,u)\rightarrow(Y,v)$ are morphisms in $\mathcal{B}$ such that $v\circ \mathsf{F}f=f\circ u$. This is an abelian category by \cite[Proposition~1.2]{marmaridis}. Note that the functor $\mathsf{F}$ induced by a cleft extension is precisely the homonymous functor of the associated $\eta$-extension.
\end{rem}

\subsection{Cleft and \texorpdfstring{$\theta$}{theta}-extensions of rings} 
The class of $\theta$-extensions was introduced in \cite{marmaridis} because it contains a large class of rings arising as extensions (see \cite{marmaridis}) and significantly simplifies complex computations. The importance of $\theta$-extensions lies in the fact that they provide a ``canonical form'' of cleft extensions of rings, as we shall explain now. 

\begin{defn} \label{def: theta}
    Let $\Gamma$ be a ring, $M$ a $\Gamma$-$\Gamma$-bimodule and $\theta\colon M\otimes_{\Gamma}M\rightarrow M$ an associative $\Gamma$-bimodule homomorphism. The \emph{$\theta$-extension of $\Gamma$ by $M$}, denoted by $\Gamma\!\ltimes\!_{\theta}M$, is defined to be the ring with underlying group $\Gamma\oplus M$ and multiplication given by
    \[
    (\gamma,m)\cdot (\gamma',m'):=(\gamma\gamma',\gamma m'+m\gamma'+\theta(m\otimes m')),
    \]
    for $\gamma,\gamma'\in\Gamma$ and $m,m'\in M$. 
\end{defn}

Observe that any $\theta$-extension of $\Gamma$ induces a cleft extension $(\Gamma \ltimes_\theta \! M, \pi, \iota)$ by defining the canonical projection $\pi\colon \Gamma \ltimes_\theta M \to \Gamma$ and inclusion $\iota\colon \Gamma \to \Gamma \ltimes_\theta M$ by $(\gamma, m) \mapsto \gamma$ and $\gamma \mapsto (\gamma, 0)$, respectively. For the converse, given a cleft extension $(\Lambda, \pi, \iota)$, set $M = \mathsf{Ker}(\pi)$ and let $\theta\colon M \otimes_\Gamma M \to M$ be the multiplication $m \otimes m' \mapsto mm'$; then $\Lambda \cong \Gamma \ltimes_\theta M$, see \cite[Section~3]{beligiannis}. Hence, the $\theta$-extension gives rise to the following cleft extension of module categories:
\[
\begin{tikzcd}
\Gamma \lMod  \arrow[rr, "\mathsf{i}"] &  & (\Gamma\!\ltimes\!_{\theta}M)\lMod  \arrow[rr, "\mathsf{e}"] \arrow[ll, "\mathsf{q}=\Gamma\otimes_{\Gamma\ltimes_{\theta}M}-"', bend right] &  & \Gamma\lMod  \arrow[ll, "\mathsf{l}=\Gamma\ltimes_{\theta} M \otimes_{\Gamma}-"', bend right]
\end{tikzcd}
\]
where $\mathsf{i}$ and $\mathsf{e}$ denote the restriction functors along the ring homomorphisms $\pi$ and $\iota$, respectively.
By Watts' theorem, the induced endofunctor $\mathsf{F}$ on $\Gamma\lMod $ is naturally isomorphic to $M \otimes_{\Gamma}-$.
If $\theta = 0$, the $\theta$-extension of rings is called \emph{trivial extension} and it is denoted simply by $\Gamma \ltimes M$; see also \cite{FGR}.

\subsection{Equivariant categories} Here we collect the needed theory of equivariant categories. For a thorough exposition of equivariant categories, we refer the reader to \cite{ChaoSun, kkp, Elagin}. Recall that throughout the paper $G$ denotes a finite group written multiplicatively.

Let $\mathcal{B}$ be an additive category.
A \emph{left action} of $G$ on $\mathcal{B}$ consists of the following data:
\begin{itemize}
\item[(i)] an additive auto-equivalence $\rho_g \colon \mathcal{B} \to \mathcal{B}$ for all $g$ in $G$, and 
\item[(ii)] a family of \emph{composition isomorphisms} $\theta_{g,h} \colon \rho_g \circ \rho_h \xrightarrow{\simeq} \rho_{gh}$ for all $g,h$ in $G$,
\end{itemize}
satisfying, for all $g,h,k$ in $G$, the following \emph{$2$-cocycle condition}:
\begin{equation*}\label{cocycle}
\begin{tikzcd}
\rho_g \circ \rho_h \circ \rho_k \arrow[rr, "\rho_g \theta_{h,k}"] \arrow[d, "\theta_{g,h} \rho_k"'] & & \rho_g\circ \rho_{hk} \arrow[d, "\theta_{g, hk}"] \\
\rho_{gh}\circ \rho_k \arrow[rr, "\theta_{gh,k}"'] & & \rho_{ghk}
\end{tikzcd}
\end{equation*}
We use the notation $(\mathcal{B},G)$ to keep track of the category along with the data of its $G$-action.

We adopt the notation ${^g \star} \coloneqq \rho_g(\star)$, where $\star$ denotes either an object or a morphism of $\mathcal{B}$. Denote by $1_G $ the identity of $G$ and note that each group action on a category admits a unique natural isomorphism $\upsilon \colon \rho_{1_G} \to \mathsf{Id}_\mathcal{B}$ satisfying the left and right unity axioms $\theta_{1_G, g} = \upsilon \rho_g$ and $\theta_{g, 1_G} = \rho_g \upsilon$ for all $g \in G$, called the \emph{unit} of the action.
If the composition isomorphisms $\theta$ are identities, we say the action is \emph{strict}. In this case, $\upsilon = \mathsf{Id}$.
Whenever there is another $G$-action on some other category, say $(\mathcal{A},G)$, we use the notation $\bar{\rho}$ for the auto-equivalences, $ \bar{\theta}$ for the composition isomorphisms and $\bar{\upsilon} \colon \bar{\rho}_{1_G} \xrightarrow{} \mathsf{Id}_{\mathcal{A}}$ for the unit.

A $G$\emph{-functor} $(\mathsf{F},\sigma)\colon (\mathcal{B},G) \to (\mathcal{A},G)$ between additive categories with $G$-actions, consists of an additive functor $\mathsf{F}\colon \mathcal{B}\to \mathcal{A}$ together with a family $\sigma$ of natural isomorphisms $\{ \sigma_g\colon\mathsf{F} \circ \rho_g \rightarrow \bar{\rho}_g \circ \mathsf{F}\}_{g\in G}$ compatible with the associativity conditions of both sides and respecting the units of both actions. That is the following two diagrams commute:
\begin{equation}
\label{G-functor}
\begin{tikzcd}
\mathsf{F}\rho_g\rho_h \arrow[d, "\mathsf{F} \theta_{g,h}"'] \arrow[r, "\sigma_{g}\rho_h"] & \bar{\rho}_g \mathsf{F} \rho_h \arrow[r,"\bar{\rho}_g \sigma_h"] & \bar{\rho}_{g}\bar{\rho}_h \mathsf{F} \arrow[d, "\bar{\theta}_{g,h} \mathsf{F}"] & \qquad & \mathsf{F} \circ \rho_{1_G} \arrow[rr, "\sigma_{1_G}"] \arrow[dr, "\mathsf{F} \circ \upsilon"'] & & \bar{\rho}_{1_G} \circ \mathsf{F} \arrow[dl, "\bar{\upsilon} \circ \mathsf{F}"] \\
\mathsf{F} \rho_{gh} \arrow[rr, "\sigma_{gh}"'] & & \bar{\rho}_{gh}\mathsf{F} & & & \mathsf{F} &
\end{tikzcd}
\end{equation}

\noindent If $\sigma_g = \mathsf{Id}$ for all $g \in G$, then we say that $\mathsf{F}$ is a \emph{strict} $G$-functor.

Let $(\mathsf{F}, \sigma^\mathsf{F})$ and $(\mathsf{H}, \sigma^\mathsf{H})$ be two $G$-functors $(\mathcal{B}, G) \to (\mathcal{A},G)$. A natural transformation $\alpha \colon  \mathsf{F}\Rightarrow \mathsf{H}$ is called a \emph{$G$-natural transformation}, denoted by $\alpha \colon  (\mathsf{F}, \sigma^\mathsf{F}) \Rightarrow_G (\mathsf{H}, \sigma^\mathsf{H})$, if it commutes with the action, that is, for all $g \in G$:
\begin{equation}\label{G-natural}
    \begin{tikzcd}
\mathsf{F} \rho_g \arrow[r, "\alpha \rho_g"] \arrow[d,"\sigma^\mathsf{F}_{g}"'] & \mathsf{H} \rho_g \arrow[d, "\sigma^\mathsf{H}_{g}"] \\
{\bar{\rho}_g}\mathsf{F} \arrow[r, "{\bar{\rho}_g} \alpha"'] & {\bar{\rho}_g}\mathsf{H}
    \end{tikzcd}
\end{equation}

\begin{defn}{\label{def:equivcats}}
Assume that $G$ acts on an additive category $\mathcal{B}$. A \emph{$G$-equivariant object} of $\mathcal{B}$ is a pair $(X , \chi) $ where $X$ belongs to $\mathcal{B}$ and $\chi$ is a family of isomorphisms $\{ \chi_g : X \xrightarrow[]{\sim} {^g X} \}_{g\in G}$ called the \emph{linearization} of $X$, such that the following diagram commutes for all $g,h$ in $G$:
    \begin{equation*}
        \begin{tikzcd}
            X \arrow[r, "\chi_g"] \arrow[rrr, bend right=20, "\chi_{gh}"'] & {^g \! X} \arrow[r, "{^g (\chi_h)}"] & {^g ( {^h \! X})} \arrow[r, "\theta_{g,h}"] & {^{gh} \! X}
        \end{tikzcd}
    \end{equation*}
The action $(\mathcal{B},G)$ gives rise to a $G$-equivariant category, denoted by $\mathcal{B}^G$, whose objects are the \linebreak $G$-equivariant objects and morphisms those that respect the linearizations, that is a morphism \linebreak $f \colon (X, \chi) \to (Y,\psi)$ is a morphism of $\mathcal{B}$ such that $\psi_g f = {^g \! f}  \chi_g$, for all $g$ in $G$.
\end{defn}

If the finite group $G$ acts on an abelian (resp.\ additive) category $\mathcal{B}$, then the equivariant category $\mathcal{B}^G$ is also abelian (resp.\ additive). Indeed, we have that $0\to (X, \chi)\to (Y, \psi)\to (Z , \zeta) \to 0$ is a short exact sequence in $\mathcal{B}^G$ if and only if $0\to X\to Y\to Z\to 0$ is a short exact sequence in $\mathcal{B}$. 
Observe that the zero object of $\mathcal{B}^G$ is the (uniquely) linearized zero object of $\mathcal{B}$. Kernels, cokernels and images of morphisms in the equivariant category are linearized kernels, cokernels and images of the morphisms in the underlying category \cite[Remark~2.6]{kkp}. Finally, isomorphic objects admit isomorphic linearized objects, thus kernels, cokernels and images admit unique linearizations \cite[Remark~2.7]{kkp}. 

\begin{rem}
\label{gfunctorinducesequivariant}
A $G$-functor $(\mathsf{F},\sigma)\colon (\mathcal{B},G) \rightarrow (\mathcal{A},G)$ induces an \emph{equivariant functor} $\mathsf{F}^G\colon \mathcal{B}^G \rightarrow \mathcal{A}^G$, by $\mathsf{F}^G (X ,\chi) \coloneqq (\mathsf{F}X, \sigma\mathsf{F}\chi)$ where $(\sigma\mathsf{F}\chi ) _g \coloneqq \sigma_{g,X}  \mathsf{F}(\chi_g)$; 
see \cite[Lemma~2.9]{kkp} for a complete proof.
If $\mathsf{F}$ is additive (resp.\ full, resp.\ faithful, resp.\ left exact, resp.\ right exact) then so is $\mathsf{F}^G$.

Similarly, a $G$-natural transformation $\alpha \colon (\mathsf{F}, \sigma^\mathsf{F}) \Rightarrow_G (\mathsf{H}, \sigma^\mathsf{H}) $ induces an \emph{equivariant natural transformation} $\alpha^G \colon  \mathsf{F}^G \Rightarrow \mathsf{H}^G$ such that $\alpha^G_{(X, \chi)} =\alpha_X$.
\end{rem}

\begin{lem}\label{lem:ker_cok_G_functor}
    Given a $G$-natural transformation $\alpha\colon (\mathsf{F},\sigma^\mathsf{F}) \Rightarrow_G (\mathsf{H}, \sigma^\mathsf{H})$ of $G$-functors between abelian categories $\mathcal{B}$ and $\mathcal{A}$, then $\Ker(\alpha)$ is naturally a $G$-functor and $\Ker (\alpha^G) \simeq (\Ker \alpha)^G$. Similarly, the same holds for the cokernel of $\alpha$.
\end{lem}
\begin{proof}
    The proof is a standard check and it is omitted.
\end{proof}

\begin{rem}
\label{compositionGfunctors}
Let $(\mathsf{H},\sigma^\mathsf{H})\colon (\mathcal{C},G) \rightarrow (\mathcal{B},G)$ and $(\mathsf{F},\sigma^\mathsf{F})\colon (\mathcal{B},G) \rightarrow (\mathcal{A},G)$ be two $G$-functors, with $\mathsf{H}^G\colon \mathcal{C}^G\rightarrow \mathcal{B}^G$ and $\mathsf{F}^G\colon \mathcal{B}^G \rightarrow \mathcal{A}^G$ the induced equivariant functors. Then $(\mathsf{F} \circ \mathsf{H}, \sigma^{\mathsf{F} \circ \mathsf{H}}) \colon (\mathcal{C},G) \to (\mathcal{A},G)$, with $\sigma^{\mathsf{F}\circ \mathsf{H}}_g \coloneqq \sigma^\mathsf{F}_g \mathsf{H} \circ \mathsf{F} \sigma^\mathsf{H}_g$ for all $g \in G$, is a $G$-functor and $(\mathsf{F}\circ \mathsf{H})^G = \mathsf{F}^G \circ \mathsf{H}^G$.
\end{rem}

\begin{rem} \label{inclusionfunctor}
Let $\mathcal{B}$ be a \emph{$G$-invariant subcategory} of $\mathcal{A}$, i.e.\ ${^g X} \in \mathcal{B}$ for all $X \in \mathcal{B}$ and $g \in G$ (where the action on $\mathcal{A}$ is given by $\bar{\rho}$). This means equivalently that $\bar{\rho}_g(\mathcal{B}) \simeq \mathcal{B}$. In this case, the action of $G$ on $\mathcal{A}$ restricts to an action of $G$ on $\mathcal{B}$. The inclusion functor $\mathcal{B} \hookrightarrow \mathcal{A}$ is naturally a strict $G$-functor (where the 2-isomorphisms $\sigma$ are identities).
Taking $\mathcal{A}=\mathcal{B}$ yields that $\mathsf{Id}_{\mathcal{A}}$ is a $G$-functor and the induced equivariant functor $\mathsf{Id}^G_{\mathcal{A}}$ is $\mathsf{Id}_{\mathcal{A}^G}$.

In general, given a fully faithful functor $\mathsf{i}\colon \mathcal{B} \to \mathcal{A} $ we can identify $\mathcal{B}$ with its essential image $\mathsf{Im}(\mathsf{i})$, and thus we can view it as a full subcategory of $\mathcal{A}$. If both categories admit $G$-actions, then $\mathcal{B}$ is a $G$-invariant subcategory of $\mathcal{A}$ if and only if $\mathsf{i}$ admits the structure of a $G$-functor; see also \cite[Lemma~2.17]{kkp}.
\end{rem}

\begin{lem}\textnormal{(\!\!\cite[Lemma~2.19]{ChaoSun})}
\label{adjointequiv}
Suppose that $(\mathsf{F}, \mathsf{H})$ is an adjoint pair, between additive categories with $G$-actions. If either of the two functors is a $G$-functor, then the other one admits a canonical $G$-functor structure, such that the unit $\upsilon \colon \mathsf{Id} \Rightarrow \mathsf{H}\mathsf{F}$ and the counit $\nu\colon  \mathsf{F}\mathsf{H} \Rightarrow \mathsf{Id}$ become $G$-natural transformations. We also have that $(\mathsf{F}^G,  \mathsf{H}^G)$ constitutes an adjoint pair with unit $\upsilon^G$ and counit $\nu^G$.
Finally, if $\mathsf{F}$ is an equivalence, then so is $\mathsf{F}^G$.
\end{lem} 

\subsection{Skew group rings}
In this subsection we discuss the well-known relationship between the equivariant category of $R$-modules and the category of modules over the skew group ring. Recall the following construction, studied extensively in \cite{RR}.

\begin{defn} \label{def: skew group ring}
Let $R$ be a ring and denote by $\mathsf{Aut}(R)$ the group of ring automorphisms of $R$.
A \emph{left $G$-action on $R$ by ring automorphisms} consists of a group homomorphism $G \to \mathsf{Aut}(R)$, such that  ${^g ({^h r})} = {^{gh}r}$, for all $g,h \in G$ and $r \in R$, where $^gr$ denotes the image of $r$ under the automorphism corresponding to $g$.
The \emph{skew group ring} $R \# G$ is defined to be the free left $R$-module with basis the elements of $G$ (i.e.\ elements are formal sums of the form $\sum_{g \in G} r_g \# g$), and multiplication defined by:
    $$ (r \#g) (r' \#g') = r \, ({^gr'}) \# g g',$$
    for all $r,r' \in R$ and $g,g' \in G$.
\end{defn}

A left $R \# G$-module is exactly a left $R$-module $M$ equipped with a \emph{compatible} left $G$-action, meaning $g \cdot (r m) = {^gr}(g \cdot m)$ for all $g \in G$, $r \in R$, and $m \in M$. In the sequel, we use also exponential notation for the $G$-action on $M$, i.e.\ ${^gm} \coloneqq g \cdot m  $.

On the other hand, $G$ induces a strict left action on the category $R \lMod$ in the following way. For a left $R$-module $M$ and $g \in G$, we denote by ${^g \! M}$ the twisted module, which has the same underlying abelian group as $M$ but is equipped with the twisted $R$-action $r \bullet_g m \coloneqq {^{g^{-1}}\!r} m$. Each $g \in G$ induces an autoequivalence $\rho_g \colon R\lMod \to R\lMod $ mapping $M \mapsto {^g \! M}$ and $f \mapsto f$. This is a strict left $G$-action on $R\lMod $, meaning all composition isomorphisms $\theta_{g,h}$ are identity natural transformations.

\begin{lem} \label{lem:skew group equivariant iso}
    There is a canonical isomorphism of categories
    $$ \Phi \colon R\#G\lMod  \xrightarrow{\cong} (R\lMod)^G $$
    which restricts to an isomorphism $R\#G\lsmod \xrightarrow{\cong} (R\lsmod)^G$.
\end{lem}

Explicitly, the functor $\Phi$ maps an $R\#G$-module $M$ to the equivariant object $(M, \chi)$, where the linearization $\chi_g \colon M \xrightarrow{\simeq} {^g\!M}$ is defined by $\chi_g(m) = {^ {g^{-1}}m}$.

\subsection{Forgetful and induction functors}

In this subsection we recall the ambidextrous adjunction $(\mathsf{Ind}, \mathsf{For}, \mathsf{Ind})$ and some of its important properties. The \emph{forgetful functor} $\mathsf{For} \colon  \mathcal{B}^G \rightarrow \mathcal{B}$ is defined by $\mathsf{For}(X, \chi) = X$ and the \emph{induction functor} $\mathsf{Ind} \colon  \mathcal{B} \rightarrow \mathcal{B}^G$ is defined by $\mathsf{Ind}(X) = \big(\bigoplus_{h \in G} {^h \! X}, \phi \big)$, where the linearization $\phi_g\colon  \bigoplus_{h\in G} {^h \! X} \rightarrow \bigoplus_{h \in G} {^g ({^h \! X})}$ is the collection of isomorphisms $\theta_{g,h}^{-1} \colon {^{gh} \! X} \rightarrow {^g ( {^h \! X} )}$. For a proof of this bi-adjunction, see \cite[Lemma~3.8]{Elagin}.

Given a $G$-functor $(\mathsf{F},\sigma)\colon (\mathcal{A},G) \rightarrow (\mathcal{B},G)$, its induced equivariant functor $\mathsf{F}^G\colon \mathcal{A}^G \rightarrow \mathcal{B}^G$ strictly commutes with the forgetful functor, i.e.\ $\mathsf{For} \circ \mathsf{F}^G = \mathsf{F} \circ \mathsf{For}$.
The equivariant functor also commutes with the induction functor up to the natural isomorphism $\oplus \sigma_g \colon  \mathsf{F}^G \circ \mathsf{Ind} \xrightarrow{\simeq} \mathsf{Ind} \circ \mathsf{F}$.

An object $X$ of $\mathcal{B}$ is always a summand of $\mathsf{For}(\mathsf{Ind}(X)) = \bigoplus_{g \in G} {^g \! X}$. When $|G|$ is invertible in $\mathcal{B}$, any linearized object $(X,\chi)$ is a retract of $\mathsf{Ind}(\mathsf{For}(X,\chi)) = \mathsf{Ind}(X)$, and thus a direct summand if $\mathcal{B}$ is idempotent complete; see for example \cite[Lemma~2.14(3)]{ChaoSun}. 

We will take advantage of the fact that both $\mathsf{Ind}$ and $\mathsf{For}$ preserve projective and injective objects.

\subsection{Projective and injective dimensions in equivariant categories}\label{sec: projectives injectives}

Assume now that $\mathcal{B}$ is abelian. Note that $\Proj \mathcal{B}$ and $\Inj \mathcal{B}$ are both $G$-invariant subcategories.
Using the forgetful-induction ambidextrous adjunction, we can easily prove that $\mathcal{B}$ has enough projectives (resp.\ injectives) if and only if $\mathcal{B}^G$ has enough projectives (resp.\ injectives).
Moreover, if $|G|$ is invertible in $\mathcal{B}$, then $\Proj \mathcal{B}^G = (\Proj \mathcal{B})^G$ and $\Inj \mathcal{B}^G = (\Inj \mathcal{B})^G$, see \cite[Lemma~3.12]{ChaoSun}. 

The following lemma, showing that projective dimension is preserved and reflected under equivariantization, is important for the study of singularity categories of cleft extensions. A version of it is used in the proof of \cite[Theorem~1.3]{RR}.

\pagebreak 

\begin{lem}\label{lem:pd_equiv}
Let $\mathcal{B}$ be an abelian category with enough projectives together with a $G$-action.
Then:
\begin{itemize}
\item[(i)] For any object $X \in \mathcal{B}$, we have $\pd_{\mathcal{B}^G}\mathsf{Ind}(X) = \pd_{\mathcal{B}}X$.
\item[(ii)] For any object $(X, \chi) \in \mathcal{B}^G$, we have $ \pd_{\mathcal{B}}X \leq \pd_{\mathcal{B}^G}(X, \chi)$. Assuming that $|G|$ is invertible in $\mathcal{B}$, we have that $ \pd_{\mathcal{B}}X = \pd_{\mathcal{B}^G}(X, \chi)$.
\end{itemize}
Dual statements hold for injective dimensions of objects, assuming $\mathcal{B}$ has enough injectives.
\end{lem}
\begin{proof}
Observe that for any $X$ in $\mathcal{B}$ we have that $\pd_\mathcal{B} X = \pd_\mathcal{B} {^g \! X}$, which yields the equality $ \pd_\mathcal{B} (\oplus_{g \in G} {^g \! X}) = \pd_\mathcal{B} X$. Indeed, $\pd_\mathcal{B} (\oplus_{g \in G} {^g \! X}) \leq \pd_\mathcal{B} X$ because the (finite) direct sum of projective resolutions which all have the same length is a projective resolution of the same length. For the converse, $\pd_\mathcal{B} X \leq \pd_\mathcal{B} \oplus_{g \in G} {^g \! X}$ is immediate from the following well known equality $\pd (\oplus_{i \in I} X_i ) = \sup_{i \in I} \{ \pd(X_i) \}$, where $I$ is any set.

(i) Applying the exact functor $\mathsf{Ind}$ to a projective resolution of $X$ yields a projective resolution of $\mathsf{Ind}(X)$ in $\mathcal{B}^G$ of length $\pd_\mathcal{B} X$, which shows that $\pd_{\mathcal{B}^G}(\mathsf{Ind}(X)) \leq \pd_{\mathcal{B}}(X)$. Conversely, applying the forgetful functor to a projective resolution of $\mathsf{Ind}(X)$ yields a projective resolution of $\mathsf{For}(\mathsf{Ind}(X)) = \oplus_{g \in G} {^g \! X}$ in $\mathcal{B}$. By the above observation, this implies that $\pd_\mathcal{B} X \leq \pd_{\mathcal{B}^G} \mathsf{Ind}(X)$.

(ii) Assume now that $X$ admits a linearization $\chi$, i.e.\ $(X,\chi) \in \mathcal{B}^G$. Applying $\mathsf{For}$ to a projective resolution of $(X,\chi)$ yields a projective resolution of $X$, hence $\pd_\mathcal{B} X \leq \pd_{\mathcal{B}^G} (X,\chi)$.
Assuming that $|G|$ is invertible in $\mathcal{B}$, then $(X,\chi)$ is a summand of $\mathsf{Ind}(X)$. Thus, we have that $\pd_{\mathcal{B}^G} (X,\chi) \leq \pd_{\mathcal{B}^G} \mathsf{Ind}(X) = \pd_\mathcal{B} X$.
\end{proof}

In fact, this lemma is interesting on its own. In order to state a nice consequence, we need the following definition.

\begin{defn}
    Let $\mathcal{B}$ be an abelian category with enough projectives and enough injectives. We say that $\mathcal{B}$ is \emph{Gorenstein} if $\mathsf{sup} \{ \pd_{\mathcal{B}}(I)  : I \in \Inj \mathcal{B} \} < \infty$ and $\mathsf{sup} \{ \id_{\mathcal{B}}(P) : P \in  \Proj \mathcal{B} \} < \infty$.
\end{defn}

The following result shows that equivariantization preserves and reflects Gorensteinness.

\begin{prop} \label{prop:Gorenstein_equiv}
    Let $\mathcal{B}$ be an abelian category with enough projectives and enough injectives. Assume a finite group $G$ acts on $\mathcal{B}$ such that $|G|$ is invertible in $\mathcal{B}$. Then $\mathcal{B}$ is Gorenstein if and only if $\mathcal{B}^G$ is Gorenstein.
\end{prop}
\begin{proof}
Using Lemma~\ref{lem:pd_equiv}, we have that $\mathsf{sup} \{ \pd_{\mathcal{B}}(I)  : I \in \Inj \mathcal{B} \} = \mathsf{sup} \{ \pd_{\mathcal{B}^G}(I,\i)  : (I,\i) \in \Inj \mathcal{B}^G \} $ and, by its dual, $\mathsf{sup} \{ \id_{\mathcal{B}}(P) : P \in  \Proj \mathcal{B} \} = \mathsf{sup} \{ \id_{\mathcal{B}^G}(P,\pi) : (P,\pi) \in  \Proj \mathcal{B}^G \}$.
\end{proof}

\begin{rem} \label{rem:modular_gorenstein}
One direction of Proposition~\ref{prop:Gorenstein_equiv} holds regardless of $|G|$ being invertible in $\mathcal{B}$: if the equivariant category $\mathcal{B}^G$ is Gorenstein, then the underlying category $\mathcal{B}$ is also Gorenstein, because Lemma~\ref{lem:pd_equiv} still provides the necessary bounds for the injective and projective dimensions.

More generally, Lemma~\ref{lem:pd_equiv} and Proposition~\ref{prop:Gorenstein_equiv} are direct consequences of the general properties of adjoint pairs $(\mathsf{F}, \mathsf{G})$ of exact functors between abelian categories with enough projectives or injectives, where one or both of them are separable (see for example \cite{ChenRen, NVV, rafael}).  In the equivariant setting, for the adjoint pairs $(\mathsf{Ind}, \mathsf{For})$ and $(\mathsf{For}, \mathsf{Ind})$ of exact functors, the functor $\mathsf{Ind}$ is always separable, and when $|G|$ is invertible in $\mathcal{B}$, the functor $\mathsf{For}$ is also separable (see \cite[Section~4]{XiaoWuChen2}).
\end{rem}

\section{\texorpdfstring{$G$}{G}-Equivariant Lifts} \label{sec: equivariant lifts}
In this section we establish sufficient conditions under which a cleft extension of abelian categories lifts to a $G$-equivariant cleft extension. We also study skew group rings of $\theta$-extensions.

\subsection{\texorpdfstring{$G$}{G}-equivariant cleft extensions}
Consider a cleft extension $(\mathcal{B},\mathcal{A},\mathsf{i},\mathsf{e},\mathsf{l})$ and a finite group $G$ acting on both $\mathcal{B}$ and $\mathcal{A}$. If the functors $\mathsf{i}$ and $\mathsf{e}$  are compatible with the group action, the entire structure of the cleft extension lifts harmoniously to the equivariant setting. We formalize this directly:

\begin{defn} \label{def: G-lift of cleft}
    A \emph{$G$-equivariant cleft extension} consists of a cleft extension of abelian categories $(\mathcal{B},\mathcal{A},\mathsf{i},\mathsf{e},\mathsf{l})$ equipped with left $G$-actions on $\mathcal{B}$ and $\mathcal{A}$, such that $\mathsf{i}$ and $\mathsf{e}$ are $G$-functors, and the natural isomorphism $\xi\colon \mathsf{e}\mathsf{i} \xrightarrow{\simeq} \mathsf{Id}_\mathcal{B}$ is a $G$-natural transformation.
\end{defn}

\begin{prop} \label{prop:equivariant_cleft_properties}
    Let $(\mathcal{B},\mathcal{A},\mathsf{i},\mathsf{e},\mathsf{l})$ be a $G$-equivariant cleft extension. Then the following hold:
    \begin{itemize} 
        \item[\textnormal{(i)}] The functors $\mathsf{l}$, $\mathsf{q}$, $\mathsf{G}$, and $\mathsf{F}$ canonically inherit the structure of $G$-functors.
        \item[\textnormal{(ii)}] The induced equivariant categories and functors form a cleft extension $(\mathcal{B}^G,\mathcal{A}^G,\mathsf{i}^G,\mathsf{e}^G,\mathsf{l}^G)$ which is part of the following commutative diagram of cleft extensions:
        
        \[
        \begin{tikzcd}
        \mathcal{B}^G \arrow[rr, "\mathsf{i}^G"'] \arrow[dd, "\mathsf{For}"'] &  & \mathcal{A}^G \arrow[rr, "\mathsf{e}^G"'] \arrow[dd, "\mathsf{For}"'] \arrow[ll, "\mathsf{q}^G"', bend right] \arrow["\mathsf{G}^G"', loop, distance=2em, in=125, out=55] &  & \mathcal{B}^G \arrow[ll, "\mathsf{l}^G"', bend right, shift right]  \arrow[dd, "\mathsf{For}"'] \arrow["\mathsf{F}^G"', loop, distance=2em, in=125, out=55]   \\
         &  & &  & \\
        \mathcal{B} \arrow[rr, "\mathsf{i}"'] &  & \mathcal{A} \arrow[rr, "\mathsf{e}"'] \arrow[ll, "\mathsf{q}"', bend right]   \arrow["\mathsf{G}"', loop, distance=2em, in=305, out=235]   &  & \mathcal{B} \arrow[ll, "\mathsf{l}"', bend right] \arrow["\mathsf{F}"', loop, distance=2em, in=305, out=235]     
        \end{tikzcd}
        \]
    \end{itemize}
\end{prop}
\begin{proof}
    For (i), since $\mathsf{e}$ and $\mathsf{i}$ are $G$-functors, their adjoints $\mathsf{l}$ and $\mathsf{q}$, respectively, admit canonical $G$-functor structures, by Lemma~\ref{adjointequiv}. Because compositions and kernels of $G$-functors yield $G$-functors, by Remark~\ref{compositionGfunctors} and Lemma~\ref{lem:ker_cok_G_functor}, the functors $\mathsf{G}$ and $\mathsf{F} \simeq \mathsf{e}\mathsf{G}\mathsf{i}$ inherit $G$-functor structures. 
    
    For (ii), by Remark~\ref{gfunctorinducesequivariant}, the $G$-natural isomorphism $\xi \colon \mathsf{e}\mathsf{i} \xrightarrow{\simeq} \mathsf{Id}_\mathcal{B}$ induces an equivariant natural isomorphism $\xi^G \colon (\mathsf{e}\mathsf{i})^G \xrightarrow{\simeq} \mathsf{Id}_{\mathcal{B}}^G$. Since $(\mathsf{e}\mathsf{i})^G = \mathsf{e}^G \mathsf{i}^G$ (Remark~\ref{compositionGfunctors}) and $\mathsf{Id}_\mathcal{B}^G = \mathsf{Id}_{\mathcal{B}^G}$ (Remark~\ref{inclusionfunctor}), we obtain the isomorphism $\mathsf{e}^G \mathsf{i}^G \simeq \mathsf{Id}_{\mathcal{B}^G}$. The adjoint pairs $(\mathsf{l}^G, \mathsf{e}^G)$ and $(\mathsf{q}^G, \mathsf{i}^G)$ are guaranteed by Lemma~\ref{adjointequiv}, confirming that the equivariant categories and functors form a cleft extension. The (strict) commutativity with $\mathsf{For}$ is a property of equivariant functors.
\end{proof}

\begin{rem}
Given a $G$-equivariant cleft extension we can show that the induced natural transformation $\eta \colon \mathsf{F}^2 \Rightarrow \mathsf{F}$ defining the associated $\eta$-extension (recall Remark~\ref{rem: eta extensions}) is a $G$-natural transformation. That is, any $G$-equivariant cleft extension is equivalent to a \emph{$G$-equivariant $\eta$-extension}.

The importance of $\eta$-extensions lies in that we can describe the extension and its properties in a manner intrinsic to the category $\mathcal{B}$. For example, we can show that  given a $G$-equivariant $\eta$-extension $\mathcal{B} \ltimes_{\eta}\mathsf{F}$, there is a canonical $G$-action on $\mathcal{B}\ltimes_{\eta}\mathsf{F}$ and a canonical isomorphism of categories:
    \[
    (\mathcal{B} \ltimes_{\eta}\mathsf{F})^G \cong \mathcal{B}^G \ltimes_{\eta^G} \mathsf{F}^G
    \]
given by $((X,u), \chi) \mapsto ((X, \chi), u)$ on objects and $f \mapsto f$ on morphisms.  
\end{rem}

\subsection{\texorpdfstring{$G$}{G}-equivariant \texorpdfstring{$\theta$}{theta}-extensions}
Recall that we use the notation ${^g \star}$ for the image of an element $\star$ under the action of $g \in G$, where $\star$ belongs to a ring or a module on which the finite group $G$ acts. If $G$ acts on rings $\Lambda$ and $\Gamma$ by ring automorphisms, a ring homomorphism $f \colon \Lambda \to \Gamma$ is called \emph{$G$-equivariant} if ${^g (f(\lambda))} = f({^g \lambda})$ for all $\lambda \in \Lambda$ and $g \in G$.

\begin{defn} \label{def: G-equivariant theta extensions}
    Assume that a finite group $G$ acts on a ring $\Gamma$ by ring automorphisms. We say that a $\Gamma$-bimodule $M$ is a \emph{$G$-compatible bimodule} if it admits a $G$-action such that $^g(\g m \g') = {^g\g} \, {^gm} \, {^g\g'}$ for all $\g, \g' \in \Gamma$, $m \in M$, and $g \in G$. 
    A $\Gamma$-bimodule homomorphism $\theta\colon M\otimes_{\Gamma}M\rightarrow M$ is a \emph{$G$-equivariant homomorphism} if ${^g \big(\theta (m \otimes m')\big)} = \theta({^gm} \otimes {^gm'})$ for all $m, m' \in M$ and $g \in G$.
    A \emph{$G$-equivariant $\theta$-extension} consists of a $\theta$-extension $\Gamma\ltimes_{\theta}M$ where $M$ is a $G$-compatible bimodule and $\theta$ is $G$-equivariant.  
\end{defn}

\begin{lem} \label{transfer of action for theta extensions}
Let $\Lambda=\Gamma \! \ltimes_\theta \! M$ be a $G$-equivariant $\theta$-extension where $G$ acts on $\Gamma$ by ring automorphisms.
Then $G$ acts on $\Lambda$ by ring automorphisms via:
\[
\tau_g(\g,m)=({^g\g},{^g m}).
\]
Moreover, the canonical maps $\pi \colon \Lambda\to\Gamma$ and $\iota \colon \Gamma\to \Lambda$ are $G$-equivariant ring homomorphisms.
\end{lem}

\begin{proof}
Each map $\tau_g$ is obviously additive and preserves the unit. For multiplicativity, let $(\g,m),(\g',m')$ be in $ \Lambda$. We have:
\begin{align*}
\tau_g\big((\g,m)(\g',m')\big) &= \tau_g(\g\g', \g m'+m\g'+\theta(m\otimes m')) \\
&= ({^g(\g\g')}, {^g(\g m'+m\g'+\theta(m\otimes m'))}) \\
&= ({^g\g} \, {^g\g'}, {^g\g} \, {^gm'} + {^gm} \, {^g\g'} + \theta({^gm} \otimes {^gm'})) \\
&= ({^g\g},{^g m})({^g\g'},{^g m'}) \\
&= \tau_g(\g, m) \tau_g(\g',m').
\end{align*}
Hence, $\tau_g$ is a ring homomorphism and, because $\tau_{g^{-1}}$ is its inverse (it is an easy check), it is an automorphism of $\Lambda$. Observe also that the equality $\tau_{gh} = \tau_g \comp \tau_{h}$ holds. The equivariance of the inclusion $\iota$ and projection $\pi$ is immediate.
\end{proof}

For a $G$-compatible $\Gamma$-bimodule $M$, the \emph{induced bimodule} (or skew group module) $M \# G$ is defined as the abelian group consisting of formal finite sums $\sum_{g\in G} m_g\# g$, where $m_g \in M$. We endow $M \# G$ with the structure of a $\Gamma \# G$-bimodule on the homogeneous elements, that is elements of $M\#G$ or of $\Gamma\#G$ of the form $m\#g$ or $\gamma\#g$ respectively, as follows:
\[ 
(\g \# g) (m \# h) (\g' \# g') = \g({^g m}) (^{gh} \g') \# ghg' 
\]
for all $\g, \g' \in \Gamma$, $m \in M$, and $g, h, g' \in G$. These actions are then extended bi-additively to all finite sums.
For a $G$-equivariant $\Gamma$-bimodule homomorphism $\theta \colon M \otimes_\Gamma M \rightarrow M$, we define the \emph{induced homomorphism} $\widehat{\theta} \colon M\# G \otimes_{\Gamma \# G} M \# G \to M \# G$ by 
\[
\widehat{\theta}\big((m\# g)\otimes (m'\# h)\big) = \theta(m\otimes {}^g m')\# gh,
\]
for all $m,m' \in M$ and $g,h \in G$, and extend it linearly to all finite sums. The verification that $\widehat{\theta}$ is an associative $\Gamma \# G$-bimodule homomorphism (whenever $\theta$ is associative) is routine.  
Hence a $G$-equivariant $\theta$-extension gives rise to the $\widehat{\theta}$-extension of $\Gamma \#G $ by $M \# G$.
Note that if $\theta=0$, then $\widehat{\theta}=0$; that is, if the $\theta$-extension is trivial, then so is the corresponding $\widehat{\theta}$-extension.
The following implies the first part of the \textbf{Main Theorem} and generalizes \cite[Proposition~3.10]{ChenWang}; see Example~\ref{ex: tensor rings} below.

\begin{thm}\label{thm: main A}
Let $\Lambda=\Gamma\ltimes_{\theta}M$ be a $G$-equivariant $\theta$-extension where $G$ acts on $\Gamma$ by ring automorphisms. Then there is an isomorphism of rings
\[
\Lambda\#G \cong (\Gamma\#G)\ltimes_{\widehat{\theta}}(M\#G).
\]
\end{thm}

\begin{proof}
By Lemma \ref{transfer of action for theta extensions}, $G$ acts by ring automorphisms on $\Lambda$, and hence we can consider the skew group ring $\Lambda \# G$. We claim that as an abelian group, the skew group ring over a direct sum splits via the map:
\[
\Psi \colon  \Lambda\#G=(\Gamma\oplus M)\#G \to(\Gamma\#G)\oplus (M\#G),  \quad \sum_{g \in G} (\gamma_g, m_g)\#g \mapsto \Big(\sum_{g \in G} \gamma_g\#g, \sum_{g \in G} m_g\#g \Big),
\]
which clearly is a group homomorphism.
For surjectivity, note that an element in $(\Gamma\#G)\oplus (M\#G)$ is of the form $\big(\sum_{g\in G} \g_g\# g, \sum_{g\in G} m_g\# g\big)$ which is precisely the image of $\sum_{g\in G} (\g_g, m_g)\# g$. For injectivity, if $\Psi\big(\sum_{g\in G} (\g_g, m_g)\# g\big) = (\sum_{g \in G}0\#g ,\sum_{g\in G }0\#g)$ 
then $\g_g = 0$ and $m_g = 0$ for all $g \in G$.

We claim that $\Psi$ is a ring isomorphism $\Lambda\#G \to (\Gamma\#G)\ltimes_{\widehat{\theta}}(M\#G)$. 
Note that it preserves the identity, since $\Psi((1_\Gamma, 0)\# 1_G) = (1_\Gamma \# 1_G, 0 \# 1_G)$, which is the identity of $(\Gamma\#G)\ltimes_{\widehat{\theta}}(M\#G)$.
It remains to show that $\Psi$ preserves multiplication.

Taking the product of two homogeneous elements in $\Lambda\#G$, we have:
\begin{align*}
    \big((\g,m)\# g\big) \big((\g',m')\# h\big) & = (\g,m)\,{}\tau_g(\g',m')\# gh \\ 
    &= (\g, m)(^g\g', ^gm')\#gh \\
    &=(\g\,{}^g\g', \ \g\,{}^g m' + m\,{}^g\g' +\theta(m\otimes{}^g m'))\# gh.
\end{align*}

Applying $\Psi$ to this product yields:
\[
\Psi\Big(\big((\g,m)\# g\big) \big((\g',m')\# h\big)\Big) =\big(\g\,{}^g\g'\#gh, \ (\g\,{}^g m' + m\,{}^g\g' +\theta(m\otimes{}^g m'))\#gh \big).
\]
On the other hand, we have that
\begin{align*}
    \Psi((\g,m)\#g) \Psi((\g',m')\#h) & = (\g\#g, m\#g) (\g'\#h, m'\#h) \\
    & = \big( (\g\#g)(\g'\#h), \ (\g\#g)(m'\#h) + (m\#g)(\g'\#h) + \widehat{\theta}((m\# g)\otimes (m'\# h)) \big) \\
    & = \big( \g{^g\g'}\#gh, \ \g{^gm'}\#gh + m{^g\g'}\#gh + \theta(m\otimes {}^g m')\# gh \big) \\
    & = \big(\g\,{}^g\g'\#gh, \ (\g\,{}^g m' + m\,{}^g\g' +\theta(m\otimes{}^g m'))\#gh \big). \qedhere
\end{align*}
\end{proof}

Recall the canonical isomorphism $\Phi_R \colon R\#G\lMod  \xrightarrow{\simeq} (R \lMod )^G$ for any ring $R$ equipped with a \linebreak $G$-action (Lemma~\ref{lem:skew group equivariant iso}).

\begin{cor} \label{cor:lifting_theta_to_cleft}
Let $\Lambda=\Gamma\ltimes_{\theta}M$ be a $G$-equivariant $\theta$-extension where $G$ acts on $\Gamma$ by ring automorphisms. Then the cleft extension of module categories induced by the $\theta$-extension $\Lambda = \Gamma  \ltimes_{\theta}  M$ lifts to a $G$-equivariant cleft extension. Furthermore, this equivariant cleft extension corresponds exactly to the cleft extension of module categories induced by the skew group ring $\Lambda\#G$. In particular, we have the following commutative diagram
\[
\begin{tikzcd}
\Gamma\#G \lMod  \arrow[rr, "\mathsf{i}\#G"] \arrow[dd, "\Phi_\Gamma"] &  & \Lambda \# G\lMod  \arrow[dd, "\Phi_\Lambda"] \arrow[rr, "\mathsf{e}\#G"] \arrow[ll, "\mathsf{q}\#G =  \Gamma \# G\otimes_{\Lambda \# G}-"', bend right] &  & \Gamma\#G\lMod  \arrow[ll, "\mathsf{l}\#G= \Lambda \# G \otimes_{\Gamma \# G} -"', bend right] \arrow[dd, "\Phi_\Gamma"] \\
\\
(\Gamma\lMod )^G \arrow[rr, "\mathsf{i}^G"] &  & (\Lambda\lMod )^G \arrow[rr, "\mathsf{e}^G"] \arrow[ll, "\mathsf{q}^G"', bend right] &  & (\Gamma\lMod )^G \arrow[ll, "\mathsf{l}^G"', bend right] 
\end{tikzcd}
\]
\end{cor}
\begin{proof}
    The upper cleft extension in the above diagram is the cleft extension induced by the $\widehat{\theta}$-extension of $\Gamma \# G$ by $M\#G$ and its isomorphism with $\Lambda \# G$ of Theorem~\ref{thm: main A}.

    By Lemma \ref{transfer of action for theta extensions}, $G$ acts on both $\Gamma$ and $\Lambda$ by ring automorphisms, hence induces strict $G$-actions on $\Gamma\lMod $ and $\Lambda\lMod $ respectively. Because the ring homomorphisms $\pi \colon \Lambda \to \Gamma$ and $\iota \colon \Gamma \to \Lambda$ are $G$\nobreakdash-equivariant, their respective induced restriction functors $\mathsf{i}$ and $\mathsf{e}$ are strict $G$-functors. Note that $\mathsf{e}\mathsf{i} \simeq \mathsf{Id}$ is a $G$-natural transformation since both are strict $G$-functors and their composition is the identity. Therefore, the cleft extension of module categories lifts to a $G$-equivariant cleft extension by Proposition~\ref{prop:equivariant_cleft_properties}, giving us the lower part of the diagram.
    
    It is a simple verification that $\Phi$ commutes with the functors of the cleft extensions. Therefore, the equivariant cleft extension of module categories over $\Gamma$ and $\Lambda$ is equivalent to the cleft extension of module categories over $\Gamma\#G$ and $\Lambda\#G$.  \qedhere
\end{proof}

\begin{rem}
We can define a \emph{$G$-equivariant cleft extension of rings} consisting of a $G$-action by ring automorphisms on $\Gamma$ and $\Lambda$ such that the ring maps $\pi\colon \Lambda \to \Gamma$ and $\iota\colon \Gamma \to \Lambda$ are $G$-equivariant.
In Lemma~\ref{transfer of action for theta extensions} we show that a $G$-equivariant $\theta$-extension of rings induces a $G$-equivariant cleft extension of rings. One can prove that the converse also holds.
\end{rem}

\begin{exmp} \label{ex: tensor rings}
Let $M$ be a $\Gamma$-bimodule. Consider the tensor ring:
\[
T_\Gamma(M) = \Gamma \oplus M \oplus M^{\otimes 2} \oplus \dots \oplus M^{\otimes n}\oplus \ldots.
\]
We view $T_\Gamma(M)$ as a $\theta$-extension of $\Gamma$ in the following way. Let $N = \bigoplus_{i=1}^{\infty} M^{\otimes i}$ denote the positive-degree part of the tensor ring. The truncated tensor product induces an associative $\Gamma$-bimodule homomorphism $\theta \colon N \otimes_\Gamma N \to N$, defined canonically by the isomorphism $M^{\otimes i} \otimes M^{\otimes j } \xrightarrow[]{} M^{\otimes i+j} $.  Therefore, the tensor ring $T_\Gamma(M)$ is exactly the $\theta$-extension $\Gamma \ltimes_\theta N$.

Assume now that $G$ acts on $\Gamma$ by ring automorphisms and $M$ is $G$-compatible. Then $M ^{\otimes i}$, for $i \geq 1$, is also $G$-compatible simply by extending the $G$-action diagonally:
\[ {^g (m_1 \otimes \ldots \otimes m_i)} \coloneqq {^g m_1} \otimes \ldots \otimes {^g m_i} \]
In this manner, $N$ is $G$-compatible and $\theta$ defined above becomes $G$-equivariant. Hence, the $\theta$-extension $\Gamma \ltimes_\theta N = T_\Gamma(M)$ is $G$-equivariant. By Theorem~\ref{thm: main A} and the fact that $T_{\Gamma\#G}(M\# G)$ is the $\widehat{\theta}$-extension of $\Gamma \#G$ by $N\# G$ we obtain the isomorphism $T_\Gamma(M) \# G \cong T_{\Gamma\#G}(M\# G)$; see also \cite[Proposition~3.10]{ChenWang}.
\end{exmp}

\section{Equivariant Derived Categories and Derived Functors}
\label{sec: equivariant derived cats}

In this section we investigate the structure of derived categories in the equivariant setting, with a particular focus on the behavior of left derived functors of right exact $G$-functors. We refer the reader to \cite{XiaoWuChen2, BeckOber, Elagin, kkp, ChaoSun} for more on this subject.

\subsection{Action on derived categories and the comparison functor}\label{actiononderivedcats}

If $\mathcal{B}$ is an abelian category, a left $G$-action on $\mathcal{B}$ extends to an \emph{admissible} action (meaning that $G$ acts by triangle auto-equivalences via the exactness of the auto-equivalences $\rho_g$) on the homotopy category $\mathsf{K}^? (\mathcal{B})$, its subcategory of acyclic complexes $\mathsf{K}^?_{ac} (\mathcal{B})$ and, on their quotient, the derived category $\mathsf{D}^?(\mathcal{B})$ by \cite[Theorem~3.14]{ChaoSun}. The action on chain complexes is computed component-wise:
\begin{equation}\label{actiononD^b}
\rho_g(\cdots \xrightarrow{} X^n \xrightarrow[]{\delta^n} X^{n+1} \xrightarrow[]{} \cdots) = \ \ \cdots \xrightarrow{} {^g \! X^n} \xrightarrow{{^g \! \delta^n}} {^g \! X^{n+1}} \xrightarrow{} \cdots 
\end{equation}
We write $ {^g(X^{\bullet})} \coloneqq \rho_g(X^{\bullet})$ for this action. 

It is natural to consider the equivariant homotopy and derived categories, $\mathsf{K}^?(\mathcal{B})^G$ and $\mathsf{D}^?(\mathcal{B})^G$, respectively. By \cite[Examples 3.19 and 3.20]{ChaoSun}, whenever $|G|$ is invertible in $\mathcal{B}$, these equivariant categories are canonically triangulated, that is, distinguished triangles in the equivariant categories are those that map via the forgetful functor to distinguished triangles in the underlying categories. Furthermore, there exists a natural equivalence up to retracts, called \emph{comparison functor}:
\[
\mathsf{K}_{\mathcal{B}} \colon  \mathsf{K}^?(\mathcal{B}^G) \xrightarrow{} \mathsf{K}^?(\mathcal{B})^G,
\]
which maps a complex of equivariant objects $(X,\chi)^{\bullet}$ to the equivariant complex $(X^{\bullet}, \chi^{\bullet})$. The comparison functor descends to the derived categories yielding the following commutative diagram 
\begin{equation}\label{eq: commutative quotient functors}
\begin{tikzcd}
    \mathsf{K}^?(\mathcal{B}^G) \arrow[r, "\mathsf{K}_{\mathcal{B}}"] \arrow[d, "\mathsf{Q}_{\mathcal{B}^G}"']  & \mathsf{K}^?(\mathcal{B})^G \arrow[d, "\mathsf{Q}_{\mathcal{B}}^G"]\\
    \mathsf{D}^?(\mathcal{B}^G) \arrow[r, "\mathsf{K}_{\mathcal{B}}"] & \mathsf{D}^?(\mathcal{B})^G
\end{tikzcd}    
\end{equation}
where $\mathsf{Q}_{\mathcal{B}^G}$ denotes the quotient functor on $\mathcal{B}^G$ and $\mathsf{Q}_\mathcal{B}^G$ denotes the equivariant quotient functor (see \linebreak \cite[Remark~2.25]{kkp} for the canonical $G$-functor structure of the quotient functor $\mathsf{Q}_\mathcal{B}$).

Note that on the level of bounded derived categories, the comparison functor is an equivalence since $\mathsf{D^b}(\mathcal{B})$ is idempotent complete by \cite[Corollary~2.10]{balmer_schlichting}. Since the auto-equivalences $\rho_g \colon \mathcal{B} \to \mathcal{B}$ preserve (complexes of) projectives, we infer that $\mathsf{K^b}(\Proj \mathcal{B})$ is a $G$-invariant full subcategory of $\mathsf{D^b}(\mathcal{B})$. The comparison functor restricts to an equivalence:
\begin{equation}\label{eq:comparison_projectives}
\mathsf{K}_{\mathcal{B}} \colon \mathsf{K^b}(\Proj \mathcal{B}^G) \xrightarrow{\simeq} \mathsf{K^b}(\Proj \mathcal{B})^G.
\end{equation}
This is in general an equivalence up to retracts but in this case, because $\Proj \mathcal{B}^G = (\Proj \mathcal{B})^G$ is an idempotent complete silting subcategory of $\mathsf{K^b}(\Proj \mathcal{B}^G)$, the latter is idempotent complete by \cite[Theorem~2.9]{IyamaYang}, and thus the comparison functor is an honest equivalence.

\subsection{Equivariant left derived functors}

We now examine the left derived functor of a right exact $G$-functor. Let $\mathsf{H}\colon \mathcal{B} \to \mathcal{A}$ be a right exact functor between abelian categories, and assume $\mathcal{B}$ has enough projectives. Then there exist the classical left derived functors $L_n\mathsf{H}\colon \mathcal{B} \to \mathcal{A}$ ($n \geq 0$) at the level of abelian categories, as well as the (total or hyper) left derived functor $\mathbb{L}\mathsf{H}\colon \mathsf{D^-}(\mathcal{B}) \to \mathsf{D^-}(\mathcal{A})$ at the level of bounded above derived categories. We say that $\mathsf{H}$ has \emph{finite (left) homological dimension}, and write $L_n\mathsf{H} = 0$ for $n$ \emph{large enough}, if there exists some $N$ such that $L_n\mathsf{H}(X) = 0$ for all $n > N$ and $X \in \mathcal{B}$. When this holds, the total left derived functor restricts to the bounded derived categories, $\mathbb{L}\mathsf{H}\colon \mathsf{D^b}(\mathcal{B}) \to \mathsf{D^b}(\mathcal{A})$. 
Dual definitions and statements hold for the right derived functors of left exact functors when $\mathcal{B}$ has enough injectives.

In the equivariant setting, assume that the categories have $G$-actions and $\mathsf{H}$ is a $G$-functor. Since $\mathcal{B}^G$ also has enough projectives (see subsection~\ref{sec: projectives injectives}) and $\mathsf{H}^G$ is right exact, the left derived functor $\mathbb{L}(\mathsf{H}^G) \colon \mathsf{D^-}(\mathcal{B}^G) \to \mathsf{D^-}(\mathcal{A}^G)$ is well-defined. 
The following proposition ensures that under these finiteness conditions and the extra assumption that $|G|$ is invertible in both categories, the restricted total derived functor $\mathbb{L}(\mathsf{H}^G) \colon \mathsf{D^b}(\mathcal{B}^G) \to \mathsf{D^b}(\mathcal{A}^G)$ (and dually for the restricted right derived functor) also exists.

\begin{prop}\label{finitenessoflocaldimension}
\textnormal{(\!\!\!\cite[Proposition~5.8(ii)]{kkp})}
Let $\mathcal{A}$ and $\mathcal{B}$ be abelian categories where $\mathcal{B}$ has enough projectives. Assume a finite group $G$ acts on both $\mathcal{A}$ and $\mathcal{B}$ such that $|G|$ is invertible in both categories. Given a right exact $G$-functor $\mathsf{H}\colon  \mathcal{B}\to\mathcal{A} $, 
we have $L_n\mathsf{H} = 0$ for $n$ large enough if and only if $L_n(\mathsf{H}^G)=0$ for $n$ large enough. The dual statement for right derived functors of left exact functors also holds.
\end{prop}

Now we turn our attention to the $G$-functor structure of the derived functor. We generalize a couple of results shown for exact functors in \cite[Section~5.1]{kkp} to left derived functors of right exact functors (and, dually, right derived functors of left exact functors).

\begin{lem} \label{lem: LF_is_G_functor}
    Let $\mathcal{A}$ and $\mathcal{B}$ be abelian categories where $\mathcal{B}$ has enough projectives. Assume a $G$-action on both $\mathcal{A}$ and $\mathcal{B}$ such that $|G|$ is invertible in both categories and let $\mathsf{H}\colon  \mathcal{B}\to\mathcal{A} $ be a right exact $G$-functor.
    Then its left derived functor $\mathbb{L}\mathsf{H} \colon \mathsf{D^-}(\mathcal{B}) \to \mathsf{D^-}(\mathcal{A})$ is canonically a $G$-functor.
    If $L_n\mathsf{H} =0$ for large enough $n$, then the latter restricts to a $G$-functor on the bounded derived categories.
    The dual statements for the right derived functors also hold.
\end{lem}

\begin{proof}
    Because $\mathcal{B}$ has enough projectives, we canonically identify $\mathsf{D}^-(\mathcal{B})$ with $\mathsf{K}^-(\Proj \mathcal{B})$ via taking projective resolutions. Under this identification, the left derived functor $\mathbb{L}\mathsf{H}$ is defined as the composition:
    \[
    \mathsf{K}^-(\Proj \mathcal{B}) \xrightarrow{\mathsf{K}\mathsf{H}} \mathsf{K}^-(\mathcal{A}) \xrightarrow{\mathsf{Q}_\mathcal{A}} \mathsf{D}^-(\mathcal{A})
    \]
    where $\mathsf{K}\mathsf{H}$ denotes the functor induced by $\mathsf{H}$ acting term-wise on chain complexes, and $\mathsf{Q}_\mathcal{A}$ is the quotient functor.
    The quotient functor $\mathsf{Q}_\mathcal{A}$ is canonically a $G$-functor (see \cite[Remark~2.25]{kkp}). Therefore, by Remark~\ref{compositionGfunctors}, it suffices to show that $\mathsf{K}\mathsf{H}$ is a $G$-functor, as the composition of $G$-functors is naturally a $G$-functor.
    
    Let $P^\bullet$ be an object in $\mathsf{K}^-(\Proj \mathcal{B})$. The $G$-functor structure of $\mathsf{H}$ provides a family of natural isomorphisms $\sigma_g \colon \mathsf{H} \circ \rho_g \xrightarrow{\simeq} \bar{\rho}_g \circ \mathsf{H}$ at the level of the abelian categories. Applying this component-wise to the chain complex yields an isomorphism:
\begin{equation}\label{eq: G-functor of homotopy}
      \sigma_g^{\mathsf{K}\mathsf{H}} \colon \mathsf{H}({^g \! P^\bullet}) \xrightarrow{\simeq} {^g (\mathsf{H}(P^\bullet))} .
\end{equation}
    Because $\mathsf{H}$ is an additive functor, it preserves chain homotopies. Thus, the component-wise isomorphism $\sigma_g^{\mathsf{K}\mathsf{H}}$ lifts to an isomorphism in $\mathsf{K}^-(\mathcal{A})$. The required compatibility conditions of diagram~(\ref{G-functor}) for $\sigma^{\mathsf{K}\mathsf{H}}$ follow immediately from those of the underlying $\sigma$. Thus, $\mathsf{K}\mathsf{H}$ is a $G$-functor and $\mathbb{L}\mathsf{H}$ becomes a $G$-functor on $\mathsf{D^-}(\mathcal{B})$ with the induced structure $\sigma^{\mathbb{L}\mathsf{H}} = \sigma^{\mathsf{Q}_\mathcal{A}} \circ \mathsf{Q}_\mathcal{A}(\sigma^{\mathsf{K}\mathsf{H}})$.
    
    Finally, if we assume $L_n\mathsf{H} = 0$ for large enough $n$, then both $\mathbb{L}\mathsf{H}$ and the natural isomorphisms $\sigma_g^{\mathbb{L}\mathsf{H}}$ restrict to the bounded derived categories.
\end{proof}

Because $\mathbb{L}\mathsf{H}$ is a $G$-functor, it induces an equivariant derived functor $(\mathbb{L}\mathsf{H})^G \colon \mathsf{D^-}(\mathcal{B})^G \to \mathsf{D^-}(\mathcal{A})^G$ by the triangulated version of Remark~\ref{gfunctorinducesequivariant}. We now prove that $(\mathbb{L}\mathsf{H})^G$ commutes with $\mathbb{L}(\mathsf{H}^G)$ up to the comparison functor.

\begin{prop} \label{prop: derived_commutation_K}
    Let $\mathsf{H}\colon \mathcal{B} \to \mathcal{A}$ be a right exact $G$-functor between abelian categories. Assume that $\mathcal{B}$ has enough projectives and $|G|$ is invertible in both categories. The following diagram commutes:
    \[
    \begin{tikzcd}
        \mathsf{D^-}(\mathcal{B}^G) \arrow[r, "\mathsf{K}_\mathcal{B}"] \arrow[d, "\mathbb{L}(\mathsf{H}^G)"'] & \mathsf{D^-}(\mathcal{B})^G \arrow[d, "(\mathbb{L}\mathsf{H})^G"] \\
        \mathsf{D^-}(\mathcal{A}^G) \arrow[r, "\mathsf{K}_\mathcal{A}"] & \mathsf{D^-}(\mathcal{A})^G
    \end{tikzcd}
    \]
    If $L_n\mathsf{H} =0$ for large enough $n$, then this diagram restricts to the bounded derived categories. The dual statement for the right derived functors also holds.
\end{prop}
\begin{proof}
    Because $\mathcal{B}$ and $\mathcal{B}^G$ have enough projectives, we can identify $\mathsf{D^-}(\mathcal{B})$ with $\mathsf{K^-}(\Proj \mathcal{B})$ and $\mathsf{D^-}(\mathcal{B}^G)$ with $\mathsf{K^-}(\Proj \mathcal{B}^G)$ via projective resolutions. Under these identifications, the derived functors decompose:
    \[ \mathbb{L}(\mathsf{H}^G) \simeq \mathsf{Q}_{\mathcal{A}^G} \circ \mathsf{K}(\mathsf{H}^G) \]
    \[ (\mathbb{L}\mathsf{H})^G \simeq (\mathsf{Q}_\mathcal{A} \circ \mathsf{K}\mathsf{H})^G \simeq \mathsf{Q}_\mathcal{A}^G \circ (\mathsf{K}\mathsf{H})^G \]
    where the bottom isomorphism is established in Lemma~\ref{lem: LF_is_G_functor}.
    The desired commutativity is equivalent to the commutativity of the following diagram:
    \[
    \begin{tikzcd}
        \mathsf{K^-}(\Proj \mathcal{B}^G) \arrow[r, "\mathsf{K}_\mathcal{B}"] \arrow[d, "\mathsf{K}(\mathsf{H}^G)"'] & \mathsf{K^-}( \Proj \mathcal{B})^G \arrow[d, "(\mathsf{K}\mathsf{H})^G"] \\
        \mathsf{K^-}(\mathcal{A}^G) \arrow[d, "\mathsf{Q}_{\mathcal{A}^G}"'] \arrow[r, "\mathsf{K}_\mathcal{A}"] & \mathsf{K^-}(\mathcal{A})^G \arrow[d, "\mathsf{Q}_{\mathcal{A}}^G"]\\
    \mathsf{D}^-(\mathcal{A}^G) \arrow[r, "\mathsf{K}_{\mathcal{A}}"] & \mathsf{D}^-(\mathcal{A})^G
    \end{tikzcd}   
    \]
    Note that the bottom square commutes by eq.~(\ref{eq: commutative quotient functors}).
    Therefore, it suffices to check that the upper square commutes, i.e.\ $\mathsf{K}_\mathcal{A} \circ \mathsf{K}(\mathsf{H}^G) \simeq (\mathsf{K}\mathsf{H})^G \circ \mathsf{K}_\mathcal{B}$. 
    
    Let $(P, \pi)^\bullet \in \mathsf{K}^{-}(\Proj \mathcal{B}^G)$. Applying $\mathsf{K}(\mathsf{H}^G)$ yields the complex of equivariant objects $(\mathsf{H}P, \sigma^\mathsf{H}\mathsf{H}\pi)^\bullet$. The comparison functor $\mathsf{K}_\mathcal{A}$ then yields the equivariant complex $((\mathsf{H}P)^\bullet, (\sigma^\mathsf{H}\mathsf{H}\pi)^\bullet)$. 
    Now, following the other direction, we have $\mathsf{K}_\mathcal{B}((P, \pi)^\bullet) = (P^\bullet, \pi^\bullet)$.
    Applying $(\mathsf{K}\mathsf{H})^G$ yields the underlying complex $(\mathsf{H}P)^\bullet$, while the linearization is given by the induced $G$-structure $\sigma^{\mathsf{K}\mathsf{H}} \circ \mathsf{H}(\pi^\bullet)$. Because $\sigma^{\mathsf{K}\mathsf{H}}$ is defined term-wise by $\sigma^\mathsf{H}$, this yields also the complex $((\mathsf{H}P)^\bullet, (\sigma^\mathsf{H}\mathsf{H}\pi)^\bullet)$. 

    For a chain map $f^\bullet \colon (P, \pi)^\bullet \to (P', \pi')^\bullet$, both paths apply $\mathsf{H}$ term-by-term, yielding the underlying chain map $\mathsf{H}(f^\bullet)$, which commutes with the linearizations because of the naturality of $\sigma^\mathsf{H}$.

    Finally, if $L_n\mathsf{H} =0$ for large enough $n$, Proposition~\ref{finitenessoflocaldimension} and Lemma~\ref{lem: LF_is_G_functor} guarantee that both $(\mathbb{L}\mathsf{H})^G$ and $\mathbb{L}(\mathsf{H}^G)$ map bounded complexes to bounded complexes. Thus, the diagram canonically restricts to the bounded derived categories.
\end{proof}

\pagebreak

\section{Equivariant Singularity Categories}\label{sec: equivariant singularity categories}

Let $\mathcal{A}$ be an abelian category with enough projectives. The \emph{singularity category} of $\mathcal{A}$, introduced in \cite{buchweitz, Orlov}, is the Verdier quotient
\[
\mathsf{D}_{\mathsf{sg}}(\mathcal{A}):=\mathsf{D}^{\mathsf{b}}(\mathcal{A})/\mathsf{K}^{\mathsf{b}}(\Proj\mathcal{A}).
\]
A triangulated functor $\mathsf{H} \colon \mathsf{D^b}(\mathcal{B}) \to \mathsf{D^b}(\mathcal{A})$ descends to a functor $\mathsf{H}_{\mathsf{sg}} \colon \mathsf{D_{sg}}(\mathcal{B}) \to \mathsf{D_{sg}}(\mathcal{A})$ if $\mathsf{H}(\mathsf{K^b}(\Proj \mathcal{B})) \subseteq \mathsf{K^b}(\Proj \mathcal{A})$. Moreover, we write that $\mathsf{H}_\mathsf{sg} \simeq 0$ in $\mathsf{D_{sg}}(\mathcal{A})$ to mean that it is isomorphic to the zero functor, which is equivalent to $\mathsf{H}( \mathsf{D^b}(\mathcal{B})) \subseteq \mathsf{K^b}(\Proj\mathcal{A})$. For a right exact functor $\mathsf{F}$, whenever $\mathbb{L}\mathsf{F}$ descends to the singularity categories, we denote the induced functor $(\mathbb{L}\mathsf{F})_{\mathsf{sg}}$ by $\mathbb{L}_{\mathsf{sg}}\mathsf{F}$.
We recall two results from \cite{kostas} that we lift to the equivariant setting. 

\begin{prop}  \label{cleft_of_singularity}
\textnormal{(\!\!\cite[Proposition~6.3]{kostas})}
        Let $(\mathcal{B},\mathcal{A},\mathsf{i},\mathsf{e},\mathsf{l})$ be a cleft extension of abelian categories with enough projectives such that the following conditions are satisfied: 
    \begin{itemize}
        \item[\textnormal{(i)}] $\pd_{\mathcal{A}}{\mathsf{i}(P)}<\infty$ for every $P\in\Proj\mathcal{B}$. 
        \item[\textnormal{(ii)}] $\pd_{\mathcal{B}}{\mathsf{e}(P)}<\infty$ for every $P\in\Proj\mathcal{A}$. 
        \item[\textnormal{(iii)}] $L_n\mathsf{l}=0$ for $n$ large enough. 
        \item[\textnormal{(iv)}] $L_n\mathsf{q}=0$ for $n$ large enough. 
    \end{itemize}
    Then we get a diagram of singularity categories and triangle functors as below 
    \[
    \begin{tikzcd}
\mathsf{D}_{\mathsf{sg}}(\mathcal{B}) \arrow[rr, "\mathsf{i}"] &  & \mathsf{D}_{\mathsf{sg}}(\mathcal{A}) \arrow[rr, "\mathsf{e}"] \arrow[ll, "\mathbb{L}_{\mathsf{sg}}\mathsf{q}"', bend right] &  & \mathsf{D}_{\mathsf{sg}}(\mathcal{B}) \arrow[ll, "\mathbb{L}_{\mathsf{sg}}\mathsf{l}"', bend right]
\end{tikzcd}
    \]
    such that $(\mathbb{L}_{\mathsf{sg}}\mathsf{l},\mathsf{e})$ and $(\mathbb{L}_{\mathsf{sg}}\mathsf{q},\mathsf{i})$ are adjoint pairs, $\mathsf{ei}\simeq \mathsf{Id}_{\mathsf{D}_{\mathsf{sg}}(\mathcal{B})}$ and $\mathbb{L}_{\mathsf{sg}}\mathsf{q}\mathbb{L}_{\mathsf{sg}}\mathsf{l}\simeq\mathsf{Id}_{\mathsf{D}_{\mathsf{sg}}(\mathcal{B})}$, where $\mathbb{L}_{\mathsf{sg}}\mathsf{l}$ denotes the functor induced by $\mathbb{L}\mathsf{l}$ and $\mathbb{L}_{\mathsf{sg}}\mathsf{q}$ the functor induced by $\mathbb{L}\mathsf{q}$.
\end{prop}

Before stating the next result, recall from subsection~\ref{subsec: cleft of abelian} that any cleft extension $(\mathcal{B}, \mathcal{A}, \mathsf{i}, \mathsf{e}, \mathsf{l})$ induces the canonical right exact endofunctor $\mathsf{F} := \mathsf{e}\mathsf{G}\mathsf{i}\colon \mathcal{B} \to \mathcal{B}$ satisfying $\mathsf{e}\mathsf{l} \simeq \mathsf{Id}_{\mathcal{B}} \oplus \mathsf{F}$.

\begin{thm}\label{thm: equivalence_singularity}
\textnormal{(\!\!\cite[Theorem~6.11]{kostas})} 
    Let $(\mathcal{B},\mathcal{A},\mathsf{i},\mathsf{e},\mathsf{l})$ be a cleft extension of abelian categories with enough projectives. Assume that the following hold: 
    \begin{itemize}
        \item[\textnormal{(a)}] $\pd_{\mathcal{B}}{\mathsf{F}(P)}<\infty$ for every $P\in\Proj\mathcal{B}$. 
        \item[\textnormal{(b)}] $L_n\mathsf{F}=0$ for $n$ large enough. 
        \item[\textnormal{(c)}] The functor $\mathsf{e}$ reflects objects with finite projective dimension.
    \end{itemize} 
    The following are equivalent: 
    \begin{itemize}
        \item[\textnormal{(i)}] The functor $\mathsf{e}\colon \mathsf{D}_{\mathsf{sg}}(\mathcal{A})\rightarrow \mathsf{D}_{\mathsf{sg}}(\mathcal{B})$ is an equivalence. 
        \item[\textnormal{(ii)}] $\mathbb{L}_{\mathsf{sg}}\mathsf{F}\simeq 0$ in $\mathsf{D}_{\mathsf{sg}}(\mathcal{B})$. 
    \end{itemize}
\end{thm}

\begin{rem}\label{rem: perfect nilpotent pdimF(X)}
    Assuming that $\mathsf{F}$ is perfect \cite[Definition~3.1]{kostas} and nilpotent (i.e.\ $\mathsf{F}^n = 0$ for some $n \geq 1$), then conditions (i)-(iv) of Proposition~\ref{cleft_of_singularity} and conditions (a)-(c) of Theorem~\ref{thm: equivalence_singularity} hold, see \cite[Corollary~6.6]{kostas} and \cite[Corollary~6.13]{kostas}, respectively, where the latter, while it is stated for cleft extensions of module categories, holds for cleft extensions of abelian categories.

    By \cite[Lemma~6.12]{kostas}, in Theorem~\ref{thm: equivalence_singularity} if instead of $\pd_\mathcal{B}\mathsf{F}(P) < \infty$ for every $P \in \Proj\mathcal{B}$, we assume that $\pd_\mathcal{B} \mathsf{F}(X)<\infty$ for all $X \in \mathcal{B}$, then  $\mathbb{L}_\mathsf{sg}\mathsf{F}\simeq 0$ in $\mathsf{D_{sg}}(\mathcal{B})$ holds and thus $\mathsf{e}$ induces a singular equivalence.
\end{rem}

We now lift the above homological properties to the equivariant setting. In general, if an exact \linebreak 
$G$-functor $\mathsf{e}\colon \mathcal{A} \to \mathcal{B}$ induces a singular equivalence and $|G|$ is invertible in both $\mathcal{A}$ and $\mathcal{B}$, then $\mathsf{e}^G$ induces a singular equivalence up to retracts; see \cite[Theorem~6.11]{kkp}. However, if $\mathsf{e}$ is part of a cleft extension, $\mathsf{e}^G$ upgrades to an honest equivalence.
Note that for $G$-equivariant cleft extensions, invertibility of $|G|$ in $\mathcal{A}$ implies its invertibility in $\mathcal{B}$ since the latter is a full subcategory of $\mathcal{A}$. 

\begin{prop} \label{prop:equivariant_conditions}
    Let $(\mathcal{B},\mathcal{A},\mathsf{i},\mathsf{e},\mathsf{l})$ be a $G$-equivariant cleft extension of abelian categories with enough projectives and $|G|$ invertible in $\mathcal{A}$. 
    \begin{enumerate}
        \item[\textnormal{(1)}] Conditions \textnormal{(i)-(iv)} of Proposition \ref{cleft_of_singularity} hold for the cleft extension if and only if they hold for the equivariant cleft extension $(\mathcal{B}^G,\mathcal{A}^G,\mathsf{i}^G,\mathsf{e}^G,\mathsf{l}^G)$.
        \item[\textnormal{(2)}] Conditions \textnormal{(a)-(c)} of Theorem \ref{thm: equivalence_singularity} hold for the cleft extension if and only if they hold for the equivariant cleft extension $(\mathcal{B}^G,\mathcal{A}^G,\mathsf{i}^G,\mathsf{e}^G,\mathsf{l}^G)$.
    \end{enumerate}
\end{prop}
\begin{proof}
    For (1): Let $(P,\pi) \in \Proj \mathcal{B}^G$. Then, $\mathsf{For}(P,\pi) = P$ is projective in $\mathcal{B}$. By condition (i) for the underlying categories, $\pd_{\mathcal{A}}\mathsf{i}(P) < \infty$. Using the identity $\mathsf{i} \circ \mathsf{For} = \mathsf{For} \circ \mathsf{i}^G$, we have that $\pd_{\mathcal{A}}\mathsf{For}(\mathsf{i}^G(P,\pi))<\infty$. Lemma \ref{lem:pd_equiv} then guarantees that $\pd_{\mathcal{A}^G}\mathsf{i}^G(P,\pi) < \infty$, proving condition~(i) for the equivariant extension. 
    Conversely, assume condition (i) holds for the equivariant setting. Let $P \in \Proj \mathcal{B}$. Since $\mathsf{Ind}$ preserves projectives, we have that $\mathsf{Ind}(P) \in \Proj \mathcal{B}^G$. By our assumption, $\pd_{\mathcal{A}^G}\mathsf{i}^G(\mathsf{Ind}(P))<\infty$. Using the natural isomorphism $\mathsf{i}^G\circ \mathsf{Ind} \simeq \mathsf{Ind} \circ \mathsf{i}$, we get $\pd_{\mathcal{A}^G}\mathsf{Ind}(\mathsf{i}(P)) < \infty$. Again, Lemma~\ref{lem:pd_equiv}, immediately implies $\pd_{\mathcal{A}}\mathsf{i}(P) < \infty$. An identical argument using $\mathsf{e}$ and $\mathsf{e}^G$, instead of $\mathsf{i}$ and $\mathsf{i}^G$, establishes the equivalence for condition (ii).
    Conditions (iii) and (iv) are equivalent to the ones in the equivariant setting by Proposition~\ref{finitenessoflocaldimension}.
    
    For (2): Condition (a) follows by exactly the same argument as (i) and (ii) above, using Lemma \ref{lem:pd_equiv} and the fact that $\mathsf{F}$ and $\mathsf{F}^G$ commute with the induction and forgetful functors. Condition (b) follows from Proposition~\ref{finitenessoflocaldimension}. 
    For condition (c), suppose $\mathsf{e}$ reflects objects of finite projective dimension. Let $(X, \chi) \in \mathcal{A}^G$ and assume $\pd_{\mathcal{B}^G}\mathsf{e}^G(X, \chi) < \infty$. Lemma \ref{lem:pd_equiv} implies that $\pd_{\mathcal{B}}\mathsf{For}(\mathsf{e}^G(X, \chi)) < \infty$, which means that $\pd_{\mathcal{B}}\mathsf{e}(X) < \infty$. Since $\mathsf{e}$ reflects this property, $\pd_{\mathcal{A}}X < \infty$. Applying Lemma \ref{lem:pd_equiv} once more yields $\pd_{\mathcal{A}^G}(X, \chi) < \infty$, showing that $\mathsf{e}^G$ reflects finite projective dimension. The converse is proven analogously using the induction functor.
\end{proof}

Assuming conditions \textnormal{(i)-(iv)} of Proposition \ref{cleft_of_singularity} hold for the  $G$-equivariant cleft extension $(\mathcal{B},\mathcal{A},\mathsf{i},\mathsf{e},\mathsf{l})$ with $|G|$ invertible in $\mathcal{A}$, we have the following diagram of singularity categories of equivariant categories and triangle functors:
    \[
    \begin{tikzcd}
\mathsf{D}_{\mathsf{sg}}(\mathcal{B}^G) \arrow[rr, "\mathsf{i}^G"] &  & \mathsf{D}_{\mathsf{sg}}(\mathcal{A}^G) \arrow[rr, "\mathsf{e}^G"] \arrow[ll, "\mathbb{L}_{\mathsf{sg}}\mathsf{q}^G"', bend right] &  & \mathsf{D}_{\mathsf{sg}}(\mathcal{B}^G) \arrow[ll, "\mathbb{L}_{\mathsf{sg}}\mathsf{l}^G"', bend right]
\end{tikzcd}
    \]
such that $(\mathbb{L}_{\mathsf{sg}}\mathsf{l}^G,\mathsf{e}^G)$ and $(\mathbb{L}_{\mathsf{sg}}\mathsf{q}^G,\mathsf{i}^G)$ are adjoint pairs, $\mathsf{e}^G\mathsf{i}^G \simeq \mathsf{Id}_{\mathsf{D}_{\mathsf{sg}}(\mathcal{B}^G)}$ and $\mathbb{L}_{\mathsf{sg}}\mathsf{q}^G\mathbb{L}_{\mathsf{sg}}\mathsf{l}^G \simeq\mathsf{Id}_{\mathsf{D}_{\mathsf{sg}}(\mathcal{B}^G)}$, where $\mathbb{L}_{\mathsf{sg}}\mathsf{l}^G$ denotes the functor induced by $\mathbb{L}\mathsf{l}^G$ and $\mathbb{L}_{\mathsf{sg}}\mathsf{q}^G$ the functor induced by $\mathbb{L}\mathsf{q}^G$.

\begin{thm} \label{thm: singularity_vanishing}
    Let $\mathsf{H}\colon \mathcal{B} \to \mathcal{A}$ be a right exact $G$-functor between abelian categories with enough projectives. Assume $|G|$ is invertible in both categories and that $\mathbb{L}_{\mathsf{sg}}\mathsf{H}$ exists. Then $\mathbb{L}_{\mathsf{sg}}(\mathsf{H}^G)$ exists and we have that:
    \[
    \mathbb{L}_{\mathsf{sg}}\mathsf{H} \simeq 0 \textnormal{ in } \mathsf{D_{sg}}(\mathcal{A}) \iff \mathbb{L}_{\mathsf{sg}}(\mathsf{H}^G) \simeq 0 \textnormal{ in } \mathsf{D_{sg}}(\mathcal{A}^G).
    \]
\end{thm}
\begin{proof}
    We first verify that $\mathbb{L}_{\mathsf{sg}}(\mathsf{H}^G)$ is well-defined. The existence of $\mathbb{L}_{\mathsf{sg}}\mathsf{H}$ means that $\mathbb{L}\mathsf{H}$ maps $\mathsf{K^b}(\Proj \mathcal{B})$ into $\mathsf{K^b}(\Proj \mathcal{A})$ and $\mathsf{D^b}(\mathcal{B})$ into $\mathsf{D^b}(\mathcal{A})$. This implies that $(\mathbb{L}\mathsf{H})^G$ maps $\mathsf{K^b}(\Proj \mathcal{B})^G$ into $\mathsf{K^b}(\Proj \mathcal{A})^G$ and $\mathsf{D^b}(\mathcal{B})^G$ into $\mathsf{D^b}(\mathcal{A})^G$. By Proposition~\ref{prop: derived_commutation_K}, the decomposition of derived functors in its proof, and because the comparison functors restrict to equivalences on the bounded homotopy categories (see eq.~(\ref{eq:comparison_projectives})), we have that $\mathbb{L}(\mathsf{H}^G)$ maps $\mathsf{K^b}(\Proj \mathcal{B}^G)$ into $\mathsf{K^b}(\Proj \mathcal{A}^G)$ and $\mathsf{D^b}(\mathcal{B}^G)$ into $\mathsf{D^b}(\mathcal{A}^G)$.

    Now, assume $\mathbb{L}_{\mathsf{sg}}\mathsf{H} \simeq 0$. This means $\mathbb{L}\mathsf{H}$ maps $\mathsf{D^b}(\mathcal{B})$ into $\mathsf{K^b}(\Proj \mathcal{A})$. Because $\mathbb{L}\mathsf{H}$ is a $G$-functor, its induced equivariant functor $(\mathbb{L}\mathsf{H})^G$ maps $\mathsf{D^b}(\mathcal{B})^G$ into $\mathsf{K^b}(\Proj \mathcal{A})^G$. As above, we infer that $\mathbb{L}(\mathsf{H}^G)$ maps $\mathsf{D^b}(\mathcal{B}^G)$ into $\mathsf{K^b}(\Proj \mathcal{A}^G)$. 
    
    Conversely, assume $\mathbb{L}_{\mathsf{sg}}(\mathsf{H}^G) \simeq 0$, that is $\mathbb{L}(\mathsf{H}^G)$ maps $\mathsf{D^b}(\mathcal{B}^G)$ into $ \mathsf{K^b}(\Proj \mathcal{A}^G)$. This is equivalent to $(\mathbb{L}\mathsf{H})^G$ mapping $\mathsf{D^b}(\mathcal{B})^G$ into $\mathsf{K^b}(\Proj \mathcal{A})^G$. Let $X^\bullet \in \mathsf{D^b}(\mathcal{B})$. We apply $\mathsf{Ind}\colon \mathsf{D^b}(\mathcal{B}) \to \mathsf{D^b}(\mathcal{B})^G$ to obtain $\mathsf{Ind}(X^\bullet) \in \mathsf{D^b}(\mathcal{B})^G$. By assumption, $(\mathbb{L}\mathsf{H})^G(\mathsf{Ind}(X^\bullet)) \in \mathsf{K^b}(\Proj \mathcal{A})^G$. Applying the forgetful functor $\mathsf{For}\colon \mathsf{D^b}(\mathcal{A})^G \to \mathsf{D^b}(\mathcal{A})$, which restricts to the bounded homotopy categories, yields that 
    $\mathsf{For}((\mathbb{L}\mathsf{H})^G(\mathsf{Ind}(X^\bullet))) \in \mathsf{K^b}(\Proj \mathcal{A})$.
    Since $\mathbb{L}\mathsf{H}$ commutes with the forgetful functor we have the following:
    \[
    \mathsf{For}((\mathbb{L}\mathsf{H})^G(\mathsf{Ind}(X^\bullet))) \simeq \mathbb{L}\mathsf{H}(\mathsf{For}(\mathsf{Ind}(X^\bullet))) \simeq \mathbb{L}\mathsf{H}\big(\bigoplus_{g \in G} {^g \! X^\bullet}\big) \simeq \bigoplus_{g \in G} \mathbb{L}\mathsf{H}({^g \! X^\bullet}).
    \]
    Since this direct sum belongs to $\mathsf{K^b}(\Proj \mathcal{A})$, which is thick, and $\mathbb{L}\mathsf{H}(X^\bullet)$ is a direct summand, it must also belong to $\mathsf{K^b}(\Proj \mathcal{A})$. Therefore, $\mathbb{L}_{\mathsf{sg}}\mathsf{H} \simeq 0$.
\end{proof}

We arrive at our main result which lifts singular equivalences to the equivariant setting.

\begin{thm} \label{thm: main B}
    Let $(\mathcal{B},\mathcal{A},\mathsf{i},\mathsf{e},\mathsf{l})$ be a $G$-equivariant cleft extension of abelian categories with enough projectives and $|G|$ invertible in $\mathcal{A}$. Assume that conditions \textnormal{(a)-(c)} of Theorem \ref{thm: equivalence_singularity} are satisfied. The following are equivalent:
    \begin{itemize}
        \item[\textnormal{(i)}] The functor $\mathsf{e} \colon \mathsf{D}_{\mathsf{sg}}(\mathcal{A}) \to \mathsf{D}_{\mathsf{sg}}(\mathcal{B})$ is an equivalence.
        \item[\textnormal{(ii)}] The functor $\mathsf{e}^G \colon \mathsf{D}_{\mathsf{sg}}(\mathcal{A}^G) \to \mathsf{D}_{\mathsf{sg}}(\mathcal{B}^G) $ is an equivalence.
    \end{itemize}
\end{thm}
\begin{proof}
    By Proposition \ref{prop:equivariant_conditions}, $\mathsf{F}$ and $\mathsf{e}$ satisfy conditions (a)-(c) if and only if $\mathsf{F}^G$ and $\mathsf{e}^G$ satisfy the same conditions on the equivariant categories. Thus, we may apply Theorem \ref{thm: equivalence_singularity} to both the underlying and the equivariant cleft extensions.
    Consequently, $\mathsf{e}$ is a singular equivalence if and only if $\mathbb{L}_{\mathsf{sg}}\mathsf{F} \simeq 0$ in $\mathsf{D}_{\mathsf{sg}}(\mathcal{B})$. Similarly, $\mathsf{e}^G$ is a singular equivalence if and only if $\mathbb{L}_{\mathsf{sg}}\mathsf{F}^G \simeq 0$ in $\mathsf{D}_{\mathsf{sg}}(\mathcal{B}^G)$. The result follows from Theorem~\ref{thm: singularity_vanishing}. 
\end{proof}

\begin{rem} \label{rem: converse without |G|}
    When $|G|$ is not invertible in $\mathcal{B}$ (and hence not in $\mathcal{A}$), the forward implication of Theorem~\ref{thm: singularity_vanishing} fails in general, and thus so does the implication $\textnormal{(i)} \implies \textnormal{(ii)}$ of Theorem~\ref{thm: main B}. However, for any right exact $G$-functor $\mathsf{H}\colon \mathcal{B} \to \mathcal{A}$, the converse implications of both Proposition~\ref{finitenessoflocaldimension} and Theorem~\ref{thm: singularity_vanishing} still hold, and the existence of $\mathbb{L}_{\mathsf{sg}}(\mathsf{H}^G)$ implies that of $\mathbb{L}_{\mathsf{sg}}\mathsf{H}$. 
    Indeed, the exact functors $\mathsf{Ind}$ and $\mathsf{For}$ preserve projectives,
    so $\mathsf{For} \circ \mathbb{L}(\mathsf{H}^G) \circ \mathsf{Ind} \simeq \mathbb{L}\mathsf{H} \circ \mathsf{For} \circ \mathsf{Ind}$ 
    on the level of derived categories 
    and similarly $\mathsf{For} \circ L_n(\mathsf{H}^G) \circ \mathsf{Ind} \simeq L_n\mathsf{H} \circ \mathsf{For} \circ \mathsf{Ind}$ 
    on the level of abelian categories. 
    Thus, $\mathbb{L}\mathsf{H}(X^\bullet)$ and $L_n\mathsf{H}(X)$ are direct summands of $\mathsf{For}(\mathbb{L}(\mathsf{H}^G)(\mathsf{Ind}(X^\bullet)))$ and $\mathsf{For}(L_n(\mathsf{H}^G)(\mathsf{Ind}(X)))$, respectively. The latter implies the converse of Proposition~\ref{finitenessoflocaldimension}, while, because $\mathsf{K^b}(\Proj \mathcal{A})$ is thick, the former implies that $\mathbb{L}_{\mathsf{sg}}\mathsf{H}$ exists whenever $\mathbb{L}_{\mathsf{sg}}(\mathsf{H}^G)$ does and yields the converse of Theorem~\ref{thm: singularity_vanishing}. 
    In other words, this argument bypasses the equivalences between the derived (and homotopy) categories of the equivariant categories and the corresponding equivariant derived (and homotopy) categories used in the proof of Theorem~\ref{thm: singularity_vanishing}.

    Applying these observations and Lemma~\ref{lem:pd_equiv}\textnormal{(i)} to the functors of the equivariant cleft extension $(\mathcal{B}^G,\mathcal{A}^G,\mathsf{i}^G,\mathsf{e}^G,\mathsf{l}^G)$, we see that if conditions \textnormal{(i)}--\textnormal{(iv)} of Proposition~\ref{cleft_of_singularity} and \textnormal{(a)}--\textnormal{(c)} of Theorem~\ref{thm: equivalence_singularity} hold for $(\mathcal{B}^G,\mathcal{A}^G,\mathsf{i}^G,\mathsf{e}^G,\mathsf{l}^G)$, then they also hold for $(\mathcal{B}, \mathcal{A},\mathsf{i},\mathsf{e},\mathsf{l})$, and the implication $\textnormal{(ii)} \implies \textnormal{(i)}$ of Theorem~\ref{thm: main B} remains valid: if $\mathsf{e}^G$ is a singular equivalence, then $\mathsf{e}$ is also a singular equivalence. In particular, if $\pd _{\mathcal{B}^G}\mathsf{F}^G(X, \chi) < \infty $ for all $(X,\chi)\in \mathcal{B}^G$, then $\pd _\mathcal{B} \mathsf{F}(X) < \infty$ for all $X \in \mathcal{B}$, and thus by Remark~\ref{rem: perfect nilpotent pdimF(X)}, we have that both $\mathsf{e}$ and $\mathsf{e}^G$ are singular equivalences.
\end{rem}

\section{Singular Equivalences of Skew Group \texorpdfstring{$\theta$}{theta}-extensions}
\label{sec: application}

In this section, we apply the theory developed in the previous sections to obtain a structural result regarding the singularity categories of skew group rings of $\theta$-extensions. Note that $|G|$ is invertible in a ring $\Gamma$ if and only if it is invertible in $\Gamma\lMod $, see \cite[Remark~6.14]{kkp}. Moreover, $|G|$ is invertible in $\Gamma$ if and only if it is invertible in any $\theta$-extension of $\Gamma$. 
Recall that, for a left noetherian ring $\Gamma$, the category $\Gamma\lsmod $ of finitely presented (equivalently, finitely generated) modules is abelian, and the singularity category $\mathsf{D_{sg}}(\Gamma)$ is defined as the Verdier quotient $\mathsf{D^b}(\Gamma\lsmod ) / \mathsf{K^b}(\Gamma\lproj )$.
Recall also that if $\Gamma$ is left (resp.\ right) noetherian, then $\Gamma \# G$ is also left (resp.\ right) noetherian.

\begin{thm} \label{thm: main application}
    Let $\Gamma$ be a left noetherian ring equipped with an action of a finite group $G$ by ring automorphisms, such that $|G|$ is invertible in $\Gamma$. Let $\Lambda = \Gamma \! \ltimes_\theta \! M$ be a $G$-equivariant $\theta$-extension and assume the following hold:
    \begin{itemize}
        \item[\textnormal{(a)}] $\fd M_\Gamma < \infty$,
        \item[\textnormal{(b)}] $\pd {_\Gamma M} <\infty$,
        \item[\textnormal{(c)}] $\Tor^\Gamma_i (M ^{\otimes j}, M)=0$ for all $i,j \geq 1$,
        \item[\textnormal{(d)}] $M$ is nilpotent \textnormal{(}i.e.\ $M^{\otimes n} =0$ for some $n \geq 1$\textnormal{)}.
        \item[\textnormal{(e)}] $M$ is finitely generated as a left $\Gamma$-module.
    \end{itemize}
    Then the following statements are equivalent:
    \begin{itemize}
        \item[\textnormal{(i)}] The restriction functor $\mathsf{e} \colon \mathsf{D}_{\mathsf{sg}}(\Lambda) \xrightarrow{} \mathsf{D}_{\mathsf{sg}}(\Gamma)$ is an equivalence.
        \item[\textnormal{(ii)}] The restriction functor $\mathsf{e}\#G \colon \mathsf{D}_{\mathsf{sg}}(\Lambda \# G) \xrightarrow{} \mathsf{D}_{\mathsf{sg}}(\Gamma \# G)$ is an equivalence.
    \end{itemize}
    Here, $\mathsf{e}\#G$ denotes the induced restriction functor of the $\widehat{\theta}$-extension $(\Gamma \# G) \ltimes_{\widehat{\theta}} (M \# G) \cong \Lambda \# G$.
\end{thm}

\begin{proof}
    Recall that the $\theta$-extension $\Lambda = \Gamma \ltimes_\theta M$ induces the cleft extension $( \Gamma\lMod, \Lambda\lMod , \mathsf{i}, \mathsf{e}, \mathsf{l})$ where the endofunctor $\mathsf{F} \colon \Gamma\lMod  \to \Gamma \lMod$ is naturally isomorphic to $M \otimes_\Gamma -$. The ring $\Gamma$ being left noetherian together with conditions (d) and (e) means that the cleft extension restricts to the cleft extension $( \Gamma\lsmod, \Lambda\lsmod , \mathsf{i}, \mathsf{e}, \mathsf{l})$ by \cite[Proposition~7.5(iv)]{beligiannis} (note that $M$ being nilpotent implies $\theta$ being nilpotent). By \cite[Lemma~3.17]{kostas}, conditions (a)-(d) are equivalent to $\mathsf{F}$ being perfect \cite[Definition~3.1]{kostas} and nilpotent. Then conditions (a)-(c) of Theorem~\ref{thm: equivalence_singularity} hold, as explained in Remark~\ref{rem: perfect nilpotent pdimF(X)}. The rest follows from Theorem~\ref{thm: main A}, Corollary~\ref{cor:lifting_theta_to_cleft} and Theorem~\ref{thm: main B}.
\end{proof}

In particular, this theorem implies the second part of the \textbf{Main Theorem}, because if $\Gamma$ is two-sided noetherian and $M$ is finitely generated on both sides, we have that $\pd{M_\Gamma} = \fd{M_\Gamma}$.

\begin{cor} \label{cor: nilpotent_application}
    Let $\Gamma$ be a two-sided noetherian ring and $M$ a $\Gamma$-bimodule that is finitely generated on both sides. Assume there is a $G$-equivariant $\theta$-extension $\Lambda = \Gamma \! \ltimes_\theta \! M$, where $G$ acts on $\Gamma$ by ring automorphisms and $|G|$ is invertible in $\Gamma$. If $M$ is a nilpotent, left-right projective $\Gamma$-bimodule \textnormal{(}that is, projective both as a left and as a right $\Gamma$-module\textnormal{)}, then the following are equivalent:
    \begin{itemize}
       \item[\textnormal{(i)}] The restriction functor $\mathsf{e} \colon \mathsf{D}_{\mathsf{sg}}(\Lambda) \xrightarrow{} \mathsf{D}_{\mathsf{sg}}(\Gamma)$ is an equivalence.
        \item[\textnormal{(ii)}] The restriction functor $\mathsf{e}\#G \colon \mathsf{D}_{\mathsf{sg}}\big((\Gamma \# G)\ltimes_{\widehat{\theta}} (M \# G)\big) \xrightarrow{} \mathsf{D}_{\mathsf{sg}}(\Gamma \# G)$ is an equivalence.
    \end{itemize}
\end{cor}
\begin{proof}
    Because $M$ is projective as a left-right bimodule,
    it satisfies conditions (a) and (b) of Theorem~\ref{thm: main application}. The higher Tor vanishing condition (c) holds since $M$ is flat on both sides.
\end{proof}

\begin{rem}
    Throughout this section we have worked over left noetherian rings so that the category of finitely generated modules coincides with the category of finitely presented modules and is abelian for both $\Gamma$ and $\Lambda = \Gamma \ltimes_\theta M$. However, Theorem~\ref{thm: main application} and Corollary~\ref{cor: nilpotent_application} admit natural generalizations to the setting of coherent rings. 
    Recall that for an arbitrary ring $R$, the category $R\lsmod$ of finitely presented left $R$-modules is non-abelian in general because it is not closed under kernels. However, it is well known that $R$ is left coherent if and only if $R\lsmod$ is abelian (for a proof see \cite{SikkoSmalo}), in which case the singularity category is defined as $\mathsf{D_{sg}}(R) := \mathsf{D^b}(R\lsmod)/\mathsf{K^b}(R\lproj)$.
    \begin{itemize}
        \item[\textnormal{(1)}] Unlike the noetherian property, if $\Gamma$ is left coherent, then the $\theta$-extension $\Lambda = \Gamma \ltimes_\theta M$ need not be left coherent in general, even for trivial extensions (see \cite[Section~2]{FGR}). However, if we assume in Theorem~\ref{thm: main application} that both $\Gamma$ and $\Lambda$ are left coherent and replace condition \textnormal{(e)} with the assumption that $M$ is finitely presented as a left $\Gamma$-module, then the cleft extension $(\Gamma\lMod, \Lambda\lMod, \mathsf{i}, \mathsf{e}, \mathsf{l})$ restricts to a cleft extension $(\Gamma\lsmod, \Lambda\lsmod, \mathsf{i}, \mathsf{e}, \mathsf{l})$ of abelian categories. 
        Indeed, $\mathsf{l} = \Lambda \otimes_\Gamma -$ and $\mathsf{q} = \Gamma \otimes_\Lambda -$ are right exact and preserve free modules of finite rank, which implies that they preserve finitely presented modules. 
        For the exact functors $\mathsf{e}$ and $\mathsf{i}$ we have that \linebreak  
        $\mathsf{e}(\Lambda) = \Gamma \oplus M \in \Gamma\lsmod$ and $\mathsf{i}(\Gamma) =\Lambda/M \in \Lambda\lsmod$, respectively, and hence they preserve finitely presented modules by the left coherence of $\Gamma$ and $\Lambda$. 
        Moreover, since $\Gamma\lsmod$ and $\Lambda\lsmod$ are abelian, so are $(\Gamma\lsmod)^G \cong (\Gamma\#G)\lsmod$ and $(\Lambda\lsmod)^G \cong (\Lambda\#G)\lsmod$ by Lemma~\ref{lem:skew group equivariant iso}, meaning that $\Gamma\#G$ and $\Lambda\#G$ are also left coherent. 
        Thus, the equivalence $\textnormal{(i)} \iff \textnormal{(ii)}$ of Theorem~\ref{thm: main application} holds by the same proof.

        \item[\textnormal{(2)}] In the setting of Corollary~\ref{cor: nilpotent_application}, suppose $\Gamma$ is assumed to be left coherent and $M$ is a nilpotent, left-right projective $\Gamma$-bimodule that is finitely generated on both sides. Then ${M} \in \Gamma\lsmod$, and since $M$ is projective as a right $\Gamma$-module, $\mathsf{F} \simeq M \otimes_\Gamma -$ is exact and preserves finitely presented modules. 
        We claim that a left $\Lambda$-module $Y$ lies in $\Lambda\lsmod$ if and only if $\mathsf{e}(Y) $ lies in $\Gamma\lsmod$. 
        The forward implication follows as in \textnormal{(1)} from the left coherence of $\Gamma$ and $M~\in~\Gamma\lsmod$. \linebreak 
        Conversely, suppose that $\mathsf{e}(Y) \in \Gamma\lsmod$. Recall from subsection~\ref{subsec: cleft of abelian} that $\e$ denotes the counit of the adjunction $(\mathsf{l},\mathsf{e})$ and $\mathsf{G}$ denotes its kernel, which satisfies $\mathsf{eG} \simeq \mathsf{Fe}$. As explained above, $\mathsf{l}$ preserves finitely presented modules. Then $\mathsf{l}\mathsf{e}(Y) \in \Lambda\lsmod$ and $\mathsf{e}\mathsf{G}(Y) \simeq \mathsf{F}\mathsf{e}(Y) \in \Gamma\lsmod$, which in turn implies that $\mathsf{l}\mathsf{e}\mathsf{G}(Y) \in \Lambda\lsmod$. Observe that $\mathsf{G}(Y)$ is a quotient of $\mathsf{l}\mathsf{e}\mathsf{G}(Y) \in \Lambda\lsmod$ via the surjection $\e_{\mathsf{G}(Y)} \colon \mathsf{leG}(Y) \twoheadrightarrow \mathsf{G}(Y)$; hence, it follows that $\mathsf{G}(Y)$ is finitely generated. Therefore, the short exact sequence $0 \to \mathsf{G}(Y) \to \mathsf{l}\mathsf{e}(Y) \xrightarrow{\epsilon_Y} Y \to 0$ expresses $Y$ as the quotient of $\mathsf{l}\mathsf{e}(Y) \in \Lambda\lsmod$ by a finitely generated submodule, which implies that $Y \in \Lambda \lsmod$. 
        Consequently, for any morphism $f$ in $\Lambda\lsmod$, the exactness of $\mathsf{e}$ and the coherence of $\Gamma$ yield $\mathsf{e}(\ker f) \simeq \ker(\mathsf{e}f) \in \Gamma\lsmod$, so $\ker f \in \Lambda\lsmod$. Thus $\Lambda\lsmod$ is abelian, which implies that $\Lambda = \Gamma \ltimes_\theta M$ is left coherent, and Corollary~\ref{cor: nilpotent_application} holds verbatim for left coherent rings $\Gamma$. The case of trivial extensions is already covered in \cite[Theorem~2.2]{FGR}.
    \end{itemize}
\end{rem}

\begin{exmp} \label{ex: cross action tensor}
    We present a concrete application of Corollary~\ref{cor: nilpotent_application} where the singularity category is non-trivial. Let $k$ be a field of characteristic different from $2$, and let $G = \mathbb{Z}/2\mathbb{Z}$. 
    
    Let $A = k[x]/(x^2)$. Because $A$ has infinite global dimension, its singularity category $\mathsf{D_{sg}}(A)$ is non-trivial. Define the base ring as the direct product $\Gamma = A \times A \times A$. The singularity category of $\Gamma$ is non-trivial since $\mathsf{D_{sg}}(\Gamma) \simeq \mathsf{D_{sg}}(A) \times \mathsf{D_{sg}}(A) \times \mathsf{D_{sg}}(A)$. We define the $G$-action on $\Gamma$ to swap the first two components and fix the third: ${^g(r_1, r_2, r_3)} = (r_2, r_1, r_3)$. 

    Let $M = (A \otimes_k A) \oplus (A \otimes_k A)$ as a $k$-vector space. We endow $M$ with a $\Gamma$-bimodule structure where the left action operates on the first two components, and the right action operates only on the third:
    \[
    (r_1, r_2, r_3) \cdot (x_1 \otimes y_1, x_2 \otimes y_2) \cdot (s_1, s_2, s_3) \coloneqq (r_1 x_1 \otimes y_1 s_3, r_2 x_2 \otimes y_2 s_3).
    \]
    The bimodule $M$ is obviously projective. Moreover, $M^{\otimes 2} =0$ because for $e_3 = (0,0,1) \in \Gamma$ and any $m, n \in M$ we have that 
    \[
    m \otimes n = (m \cdot e_3) \otimes n = m \otimes (e_3 \cdot n) = m \otimes 0.
    \]
    We define the $G$-action on $M$ by ${^g(x_1 \otimes y_1, x_2 \otimes y_2)} = (x_2 \otimes y_2, x_1 \otimes y_1)$. This action is compatible with the group action on $\Gamma$, making the trivial extension $\Lambda = \Gamma \! \ltimes \! M$ a $G$-equivariant one. 
    
    Because $M$ is both nilpotent and projective as a bimodule, by Corollary~\ref{cor: nilpotent_application}, the restriction functor induces a singular equivalence $\mathsf{D_{sg}}(\Lambda) \xrightarrow{\simeq} \mathsf{D_{sg}}(\Gamma)$ if and only if the restriction functor of the skew group rings induces a singular equivalence $\mathsf{D}_{\mathsf{sg}}(\Lambda \# G) \xrightarrow{\simeq} \mathsf{D}_{\mathsf{sg}}(\Gamma \# G)$.
    Moreover, since $\pd {_\Gamma M _\Gamma} = 0$ we have that these singular equivalences exist by \cite[Lemma~3.5]{Dalezios} and \cite[Corollary~6.13]{kostas}.
\end{exmp}

\begin{exmp} \label{ex: equivariant tensor ring}
    We return to a special case of Example~\ref{ex: tensor rings}.
    Let $M$ be a nilpotent $\Gamma$-bimodule, such that $M^{\otimes n+1} = 0$ for some integer $n \ge 1$. Consider the (truncated) tensor ring:
    \[
    T_\Gamma(M) = \Gamma \oplus M \oplus M^{\otimes 2} \oplus \dots \oplus M^{\otimes n},
    \]
    which is the $\theta$-extension $\Gamma \ltimes_\theta N$, where $N = \oplus_{i=1}^n M^{\otimes i}$. Assuming that $G$ acts on $\Gamma$ by ring automorphisms and $M$ is $G$-compatible, then $T_\Gamma(M) \# G \simeq T_{\Gamma\#G}(M\# G)$.

    Assume now that $\Gamma$ is a finite dimensional algebra over a field $k$ and $M$ is finitely generated on both sides. Then $T_\Gamma(M)$ is also a finite dimensional algebra over $k$ and $N$ is finitely generated on both sides.
    Assume further that $|G|$ is invertible in $k$, $\pd {_\Gamma M} < \infty$, $\pd {M _\Gamma} < \infty$, and $\Tor^\Gamma_i (M^{\otimes j} , M)=0$ for all $i,j \geq 1$. By \cite[Lemma~4.5]{ChenLu}, we have that $\pd{_\Gamma N}< \infty$,  $\pd{N_{\Gamma}}< \infty$ and $\Tor^{\Gamma}_i (N^{\otimes j} , N)=0$ for all $i,j \geq 1$. 
    By Theorem~\ref{thm: main application}, the restriction functor induces an equivalence $\mathsf{D_{sg}}(T_\Gamma(M)) \to \mathsf{D_{sg}}(\Gamma) $ if and only if the corresponding restriction functor induces an equivalence $  \mathsf{D_{sg}}(T_{\Gamma \#G}(M \# G)) \to \mathsf{D_{sg}}(\Gamma \#G)$.

    If we assume that $\pd{ _\Gamma M _\Gamma} < \infty$, then $\pd{ _\Gamma N _\Gamma} < \infty$ and we have that these are both singular equivalences by \cite[Lemma~3.5]{Dalezios} and \cite[Corollary~6.13]{kostas}.
\end{exmp}


\begin{thebibliography}{99}

\bibitem{balmer_schlichting} 
    \textsc{P.~Balmer, M.~Schlichting}, \emph{Idempotent completion of triangulated categories}, J.\ Algebra 236 (2001), no.\ 2, 819--834.

\bibitem{beligiannis}
   \textsc{A.~Beligiannis}, \emph{Cleft extensions of abelian categories and applications to ring theory}, Comm.\ Algebra 28 (2000), no.\ 10, 4503--4546.

\bibitem{beligiannis2}
    \textsc{A.~Beligiannis}, \emph{On the relative homology of cleft extensions of rings and abelian categories}, J.\ Pure Appl.\ Algebra 150, (2000), pp.\ 237--299.

\bibitem{BeckOber}
\textsc{T.~Beckmann, G.~Oberdieck}, {\em On equivariant derived categories},  Eur.\ J.\ Math.\ 9 (2023), no.\ 2, Paper No.\ 36, 39 pp.
    
\bibitem{buchweitz}
     \textsc{R.-O.~Buchweitz}, \emph{Maximal Cohen–Macaulay modules and Tate-cohomology over Gorenstein rings}, volume 262 of Mathematical Surveys and Monographs. American Mathematical Society, Providence, RI, [2021] \copyright 2021.



\bibitem{XiaoWuChen2}
\textsc{X.W.~Chen}, {\em A note on separable functors and monads with an application to equivariant derived categories}, Abh.\ Math.\ Semin.\ Univ.\ Hambg.\ 85 (2015), no.\ 1, 43--52.

\bibitem{ChenWang}
\textsc{X.W.~Chen, R.~Wang},
{\em Preprojective algebras, skew group algebras and Morita equivalences}, J.\ Pure Appl.\ Algebra 229 (2025), no.~12, Paper No. 108141, 24 pp.


\bibitem{ChenLu}
\textsc{X.W.~Chen, M.~Lu}, {\em Gorenstein homological properties of tensor rings}, Nagoya Math.\ J.\ 237 (2020), 188--208.

\bibitem{ChenRen}
    \textsc{X.W.~Chen, W.~Ren}, {\em Frobenius functors and Gorenstein homological properties}, J.\ Algebra 610 (2022), 18–37.

\bibitem{Dalezios}
\textsc{G.~Dalezios}, \emph{On singular equivalences of Morita type with level and Gorenstein algebras}, Bull.\ Lond.\ Math.\ Soc.\ 53 (2021), no.~4, 1093--1106.

\bibitem{Demonet}
\textsc{L.~Demonet}, {\em Categorification of skew-symmetrizable cluster algebras}, Algebr.\ Represent.\ Theory, 14 (6), 1087--1162, 2011.

\bibitem{Elagin}
    \textsc{A.~Elagin}, \emph{On equivariant triangulated categories}, arXiv:1403.7027. 

\bibitem{FGR}
    \textsc{R.~M. Fossum, P.~A. Griffith, I.~Reiten}, \emph{Trivial extensions of abelian categories}, Lecture Notes in Mathematics, Vol.\ 456, Springer, Berlin-New York, 1975.

\bibitem{IyamaYang}
\textsc{O.~Iyama, D.~Yang}, \emph{Silting reduction and Calabi-Yau reduction of triangulated categories}, Trans.\ Amer.\ Math.\ Soc.\ 370 (2018), no.~11, 7861--7898.

\bibitem{kkp} 
     \textsc{M.~Karakikes, A.~Kontogeorgis, C.~Psaroudakis}, \emph{Equivariant recollements and singular equivalences}, arXiv:2504.07620. 

\bibitem{kk} 
     \textsc{M.~Karakikes, P.~Kostas}, \emph{Homological aspects of separable Extensions of triangulated categories}, arXiv:2604.17996. 

\bibitem{kostas}
    \textsc{P.~Kostas}, \emph{Cleft extensions of rings and singularity categories}, J.\ Algebra 685 (2026), 160--224.

\bibitem{marmaridis}
    \textsc{N.~Marmaridis}, \emph{On extensions of abelian categories with applications to ring theory}, J.\ Algebra 156 (1993), no.~1, 50--64.


\bibitem{Mumford}
\textsc{D.~Mumford}, {\em Abelian Varieties}, With appendices by C.\ P.\ Ramanujam and Yuri Manin. Corrected reprint of the second (1974) edition. Tata Institute of Fundamental Research Studies in Mathematics, 5.\ Published for the Tata Institute of Fundamental Research, Bombay; by Hindustan Book Agency, New Delhi, 2008.\ xii+263 pp.

\bibitem{git}
\textsc{D.~B.~Mumford, J.~Fogarty}, {\it Geometric invariant theory}, second edition, 
Ergebnisse der Mathematik und ihrer Grenzgebiete, 34, Springer, Berlin, 1982.

\bibitem{NVV} 
\textsc{C.~N\u{a}st\u{a}sescu, M.~Van den Bergh, F.~Van Oystaeyen}, \emph{Separable functors applied to graded rings}, J.\ Algebra 123 (1989), no.\ 2, 397--413.
 

\bibitem{Orlov}
\textsc{D.~O.~Orlov},
\textit{Triangulated categories of singularities and D-branes in Landau-Ginzburg models}, Tr.\ Mat.\ Inst.\ Steklova 246 (2004), Algebr.\ Geom.\ Metody, Svyazi i Prilozh., 240--262; translation in Proc.\ Steklov Inst.\ Math.\ 2004, no. 3(246), 227--248.

\bibitem{PR}
\textsc{I. Palm\'er, J.-E. Roos}, {\em Explicit formulae for the global homological dimensions of trivial extensions of rings}, J.\ Algebra 27 (1973), 380--413.

\bibitem{rafael} 
\textsc{M.~D.~Rafael}, \emph{Separable functors revisited}, Comm.\ Algebra 18 (1990), no.\ 5, 1445--1459.
    
\bibitem{RR}
    \textsc{I.~Reiten, C.~Riedtmann},
    {\em Skew group algebras in the representation theory of artin algebras},  J.\ Algebra 92 (1985), no.\ 1, 224--282.

\bibitem{SikkoSmalo}
\textsc{S.~A.~Sikko, S.~O.~Smal\o},
\emph{Coherent rings and homologically finite subcategories},
Math.\ Scand.\ 77 (1995), no.\ 2, 175--183.


\bibitem{ChaoSun}
\textsc{C.~Sun}, {\em A note on equivariantization of additive categories and triangulated categories}, J.\ Algebra 534 (2019), 483--530.  



\end{thebibliography}
\end{document}